\documentclass[twoside, 11pt]{article}

\usepackage[abbrvbib, hyperref, preprint]{jmlr2e}

\usepackage[utf8]{inputenc} 
\usepackage{url}            
\usepackage{booktabs}       
\usepackage{nicefrac}       
\usepackage{microtype}      
\usepackage[shortlabels]{enumitem}

\usepackage[parfill]{parskip}
\usepackage
{
	bbm,
	graphicx,
	amssymb,
	amsmath,
	dsfont, 
	xcolor,
	enumitem,
	mathtools,
	ifthen,
	mdframed,
	multirow,
	comment,
    footmisc,
    breakcites,
    mathrsfs
}

\setcitestyle{authoryear, aysep={ }, comma}

\newcommand{\new}[1]{#1}
\newcommand{\rev}[1]{\textcolor{black}{#1}}

\newcommand*{\QED}[1][$\blacksquare$]{%
\leavevmode\unskip\penalty9999 \hbox{}\nobreak\hfill
    \quad\hbox{#1}%
}

\newcommand{\N}{\mathbb{N}}                                     
\newcommand{\R}{\mathbb{R}}                                     
\newcommand{\C}{\mathbb{C}}                                     
\newcommand{\Z}{\mathbb{Z}}                                     

\newcommand{\dd}{\mathrm{d}}                                    

\providecommand{\abs}[1]{\left\lvert #1 \right\rvert}           
\providecommand{\norm}[1]{\left\lVert #1 \right\rVert}          
\newcommand{\innerprod}[2]{\left\langle #1,\, #2 \right\rangle} 
\newcommand{\tr}[1]{\mathrm{Tr} \left( #1 \right)}              

\newcommand\restr[2]{{\left.\kern-\nulldelimiterspace #1 \vphantom{\big|} \right|_{#2}}} 

\DeclareMathOperator{\domain}{dom}
\DeclareMathOperator{\range}{range}

\DeclareMathOperator{\acc}{acc}
\DeclareMathOperator{\dist}{dist}
\DeclareMathOperator*{\esssup}{ess\,sup}

\DeclarePairedDelimiter\floor{\lfloor}{\rfloor}

\newcommand{\Ex}{\mathbb{E}}
\renewcommand{\Pr}{\mathbb{P}}

\newcommand{\sigalg}{\mathcal{F}}
\newcommand{\setM}{\mathcal{M}}

\newcommand{\law}{\mathcal{L}}

\newcommand{\Cov}{\mathrm{Cov}}

\newcommand{\bounded}[1]{\mathrm{L}(#1)}
\newcommand{\HS}[1]{S_2(#1)}

\newcommand{\lag}{\eta}

\newcommand{\inrkhs}{{\mathscr{H}}} 

\newcommand{\id}[1][]{
	\ifthenelse{\equal{#1}{}}{{I}}{{I}_{\scriptscriptstyle #1}}}

\newcommand{\idop}[1][]{
	\ifthenelse{\equal{#1}{}}{{\mathcal{I}}}{{\mathcal{I}}_{\scriptscriptstyle #1}}}

\newcommand{\cov}[1][]{C_{#1}} 

\newcommand{\ko}[1][]{
   \ifthenelse{\equal{#1}{}}{\mathcal{K}}{\mathcal{K}_{#1}}}
\newcommand{\eko}[1][]{
   \ifthenelse{\equal{#1}{}}{\widehat{\mathcal{K}}}{\widehat{\mathcal{K}}_{#1}}}

\DeclareMathOperator{\mspan}{span}

\mathtoolsset{centercolon}

\newcommand{\emu}{\widehat{\mu}}

\allowdisplaybreaks

\newtheorem{assump}[theorem]{Assumption}

\usepackage{lastpage}
\jmlrheading{23}{2022}{1-\pageref{LastPage}}{5/20; Revised
8/22}{10/22}{20-442}{%
Mattes Mollenhauer, Stefan Klus, Christof Sch\"{u}tte and P\'{e}ter Koltai}
\ShortHeadings{Kernel Autocovariance Operators}{%
Mollenhauer, Klus, Sch\"{u}tte and Koltai}

\begin{document}
  
\title{Kernel Autocovariance Operators of Stationary Processes:
Estimation and Convergence}

\author{\name Mattes Mollenhauer\textsuperscript{$1$}
        \email mattes.mollenhauer@fu-berlin.de\\
        \name Stefan Klus\textsuperscript{$3$}
        \email s.klus@hw.ac.uk\\
        \name Christof Sch\"{u}tte\textsuperscript{$1$,$2$}
        \email schuette@zib.de \\
        \name P\'{e}ter Koltai\textsuperscript{$1$}
        \email peter.koltai@fu-berlin.de\\
        \textsuperscript{$1$}\addr
        Institute of Mathematics, Freie Universität Berlin \\
        Arnimallee 6, D-14195 Berlin \\
        \textsuperscript{$2$}\addr
        Zuse Intitute Berlin \\
        Takustraße 7, D-14195 Berlin \\
        \new{%
        \textsuperscript{$3$}\addr
        School of Mathematical \& Computer Sciences,
        Heriot--Watt University \\
        Edinburgh, EH14 4AS, UK}
        }

\editor{Ingo Steinwart}
\maketitle

\begin{abstract}%
    We consider autocovariance operators
    of a stationary stochastic process on a Polish space
    that is embedded into a reproducing kernel Hilbert space.
    We investigate how empirical estimates of these operators
    converge along realizations of the process under various conditions.
    In particular, we examine ergodic and strongly mixing processes and
    obtain several asymptotic results as well as
    finite sample error bounds.
    We provide applications of our theory 
    in terms of consistency results
    for \emph{kernel PCA} with dependent data and
    the \emph{conditional mean embedding} of transition probabilities.
    Finally, we use our approach to examine
    the nonparametric estimation of Markov transition operators
    and highlight how our theory 
    can give a consistency analysis for a large family of spectral analysis methods
    including kernel-based \emph{dynamic mode decomposition}.
\end{abstract}


\section{Introduction}

The \emph{kernel mean embedding},
i.e., the embedding of a probability
distribution into a \emph{reproducing kernel Hilbert space}
\citep{Berlinet04:RKHS,Smola07Hilbert},
and the closely related theory of 
\emph{kernel covariance operators} 
have spawned a vast variety of 
nonparametric models and statistical tests over the last years.
For an overview of the kernel mean embedding theory, 
we refer the reader to the survey by~\citet{MFSS16} 
and the references therein.
Kernel covariance operators serve as the theoretical foundation
of several spectral analysis and component decomposition techniques including 
\emph{kernel principal component analysis}, \emph{kernel independent component analysis} and 
\emph{kernel canonical correlation analysis}.
Consistency results and the statistical analysis of these methods can therefore be directly based
on the estimation of kernel covariance operators~\citep{Blanchard2007,Fukumizu07:KCCA,RBD10}.
Moreover, kernel covariance operators and their connection to
$L^p$-space integral operators and random matrices are a
fundamental concept used to formalize
statistical learning~\citep[see for example][]{Smale2007,RBD10}.
 
In this paper, we extend the statistical theory
of kernel covariance operators from the independent scenario to 
\emph{kernel autocovariance operators}
of a stationary stochastic process (that is, kernel cross-covariance operators with 
respect to a time-lagged version of the process).
Recently, several nonparametric models
for dependent data, sequence modeling,
and time series analysis based on kernel mean embeddings have emerged.
Popular approaches include
filtering~\citep{SHSF09,Fukumizu13:KBR,Gebhardt2019Kalman},
transition models~\citep{SunEtAl2019, GrunewalderLBPGJ2012}
and reinforcement learning~\citep{VanHoof15,LeverEtAl:CCME17,VanHoof17,Stafford:ACCME18, Gebhardt2019Robot}, to only name a few.
A theoretical tool to understand these concepts is the 
kernel autocovariance operator, as its plays an important
role in the nonparametric approximation of transition probabilities.
This concept has been introduced in an operator-theoretic sense
under the name
\emph{conditional mean embedding} by~\cite{SHSF09}
under strong technical requirements
\citep[see also][]{Klebanov2019rigorous}.
These requirements have later been relaxed by developing the theory
in a vector-valued regression 
scenario~\citep{Gruen12,Park2020MeasureTheoretic, 
Mollenhauer2020Nonparametric, LiEtAl2022}.
Although time series are one of the primary fields of application,
consistency results for the empirical conditional mean embedding
have been limited to the case of independent data until now.
As an application of the results of this paper, we prove
standard consistency statements for dependent data.

Recent results ~\citep{KBSS18,Klus2019, Mollenhauer2020Nonparametric} 
show that eigenfunctions of \emph{Markov transition operators}
can be approximated with the conditional mean embedding.
In particular, it was discovered that a large family of kernel-based
spectral analysis and model order reduction techniques for stochastic processes and dynamical systems
(see \citealp{KBBP16,KNKWKSN18} and \citealp{Wu2020} for an overview)
implicitly approximate the spectral decomposition of a transition operator
defined on an RKHS.
This operator can be expressed
in terms of kernel autocovariance operators.
Different versions of these methods are popular in fluid dynamics~\citep{Schmid10,TRLBK14,WKR15,WRK15:Kernel}, signal processing~\citep{MS94}, machine learning~\citep{HZKM03,Kawahara2016,Hua2017},
and molecular dynamics~\citep{PPGDN13,SP15} under the names
\emph{dynamic mode decomposition} and \emph{time-lagged independent component analysis}.
Until now, a full statistical convergence analysis of these 
techniques has not been conducted to the best of our knowledge.
A theoretical examination of kernel autocovariance operators
contributes significantly to the understanding of kernel-based versions of the aforementioned approaches.

The theory of weakly dependent random processes taking
values in infinite-dimensional Banach spaces or Hilbert spaces
has become increasingly important especially due to applications in 
the field of \emph{functional data analysis}~\citep{HoermannKokoszka2010,HorvathKokoszka2012}.
In infinite-dimensional statistics, the estimation of covariance and cross-covariance
operators~\citep{Baker70:XCov,Baker1973} is a fundamental concept.
Under parametric model assumptions about the process, the estimation of 
covariance and autocovariance operators has been examined in various scenarios.
For autoregressive (AR) processes in Banach spaces and Hilbert spaces,
weak convergence and asymptotic normality has been established
\citep{Bosq2000,Bosq2002,Mas2002,DehlingSharipov2005, Mas2006}.
\cite{Soltani2011} add the assumption of periodic
correlation for AR processes in Hilbert spaces. 
\cite{Allam2014,Allam2019} provide rates
for almost sure convergence of covariance operators 
in Hilbert--Schmidt norm for an AR process with random
coefficients.
For processes in an $L^2$ function space, 
the weak convergence of covariance operators has been examined
by~\cite{KokoszkaReimherr2013} under the 
assumption of $L^4$-$m$ \emph{approximability}~\citep[a concept generalizing
$m$-\emph{dependence}, which includes 
certain autoregressive and nonlinear models, see][]{HoermannKokoszka2010} 
in the context of functional principal component analysis. 

In this paper, we consider
a stationary stochastic process $(X_t)_{t \in \Z}$ 
taking values in a Polish space $E$. 
Let $\inrkhs$ be a reproducing kernel Hilbert space (RKHS)
with the \emph{canonical feature map}
$\varphi: E \rightarrow \inrkhs$.
In contrast to the previously mentioned work,
we investigate autocovariance operators of the 
corresponding embedded RKHS-valued process
\begin{equation*}
\big( \varphi(X_t) \big)_{t \in \Z}.
\end{equation*} 
We face the challenge that
properties of the $E$-valued process  $(X_t)_{t \in \Z}$ 
which quantify the convergence speed of any empirical statistic 
must transfer accordingly to
the embedded version of the process in the Hilbert space $\inrkhs$.
However, we can not require
restrictive assumptions about the feature map $\varphi$
in order to ensure applicability of our results
for various RKHSs in practice.
Hence, in contrast to the previously mentioned literature, we
consider a more general setting that does not require
either $(X_t)_{t \in \Z}$ or $(\varphi(X_t))_{t \in \Z}$
to obey any specific parametric time series model.

Recently, \cite{BZ2019} derived a Bernstein-type inequality for Hilbert space
processes for a class of mixing properties called 
$\mathcal{C}$-mixing~\citep{Maume-Deschamps2006}.
As a special case,
the authors show that under 
fairly restrictive Lipschitz conditions on the feature map $\varphi$,
this mixing property is preserved under the RKHS embedding of a so-called $\tau$-mixing process.
The derived inequality is then used to obtain concentration bounds for 
the context of RKHS learning theory, including kernel covariance operator estimation 
\emph{without a time lag}.
As described for example by \cite{HangSteinwart2017}, the 
class of $\mathcal{C}$-mixing coefficients is only 
partly related to the classical strong mixing coefficients 
found in the literature~\citep{Doukhan94,Bradley2005},
which we will consider in this paper, in particular 
the concept of $\alpha$\emph{-mixing}.

The contributions of this paper are:
\begin{enumerate}[label=(\roman*)]
    \item A mathematical framework for kernel autocovariance operators
        of a stationary discrete-time process $(X_t)_{t \in \Z}$
        taking values in a Polish space.
        This framework allows the 
        investigation of ergodicity and strong mixing in the 
        context of the embedded
        process $(\varphi(X_t))_{t \in \Z}$
        under minimal requirements on the feature map $\varphi$.
        In particular, our assumptions are easily justifiable
        for further application in practical work on RKHS-based time series models.
        \newpage
    \item A collection of various asymptotic as well as nonasymptotic
        results about the estimation error
        of empirical kernel autocovariance operators
        based on single trajectories of the process $(X_t)_{t \in \Z}$.
        These results are presented
        in a form that is directly accessible for work on related topics.
    \item Applications of our results to
        \begin{enumerate}
            \item the consistency of kernel PCA with dependent data;
            \item the consistency of the conditional mean embedding
                of transition probabilities under the typical technical assumptions; and
            \item the estimation of Markov transition operators
                and their role in a family of spectral analysis methods
                for dynamical systems.
        \end{enumerate}
\end{enumerate}

This paper is structured as follows.
In Section~\ref{sec:preliminaries}, we recall the required preliminaries
from spectral theory, Bochner integration, and reproducing kernel
Hilbert spaces and formulate our working assumptions.
Section~\ref{sec:SLLN} addresses the strong law of large numbers
of empirical kernel autocovariance operators
under the hypothesis of ergodicity.
We introduce the concept of strong mixing and derive standard probabilistic
results including the central limit theorem in 
Section~\ref{sec:asymptotic_error}.
A general concentration bound for the estimation error
can be found in Section~\ref{sec:concentration_bounds}.
Based on these results,
we highlight applications to kernel PCA from dependent data
(Section~\ref{sec:kernel_pca}), the conditional mean embedding
(Section~\ref{sec:CME}) and the approximation of
Markov transition operators (Section~\ref{sec:markov_operators}).
We conclude our work in Section~\ref{sec:conclusion}.

\section{Preliminaries}
\label{sec:preliminaries}
\subsection{General Notation}

We give an overview of our notation and collect
well-known facts from operator theory and probability theory.
For details, we refer the reader to~\citet{Reed} and~\citet{Kallenberg}.
In what follows, we write $B$ for a separable real Banach space 
with norm $\norm{\cdot}_B$, and $H$ for a separable real Hilbert space with inner product
$\innerprod{\cdot}{\cdot}_H$. 
$\bounded{B}$ stands for the Banach space of bounded 
linear operators on $B$ equipped with the operator norm $\norm{\cdot}$.
The expression $H \otimes H$
denotes the tensor product space:
$H \otimes H$ is the Hilbert space completion 
of the algebraic tensor product with respect to the
inner product 
$\innerprod{x_2 \otimes x_1}{x_2' \otimes x_1'}_{H \otimes H} = 
\innerprod{x_1}{x_1'}_H \innerprod{x_2}{x_2'}_H$. 
For $x_1,x_2 \in H$, we interpret the element $x_2 \otimes x_1 \in H \otimes H$ 
as the linear
rank-one operator $x_2 \otimes x_1 \colon H \rightarrow H$
defined by $x \mapsto \innerprod{x}{x_1}_H x_2$.
Whenever $(e_i)_{i \in I}$ is a complete orthonormal system (CONS) in $H$,
$(e_i \otimes e_j)_{i,j \in I}$ is a CONS in $H \otimes H$.
Thus, when $H$ is separable, $H \otimes H$ is separable.

Every compact operator $A$ on $H$ admits a \emph{singular value decomposition},
that is, there exist orthonormal systems $\{u_i\}_{i\in J}$ and $\{v_i\}_{i\in J}$ in $H$ such that
\begin{equation} \label{eq:SVD}
    A = \sum_{i \in J} \sigma_i(A) u_i \otimes v_i,
\end{equation}
where $(\sigma_i(A))_{i \in J}$ are the strictly positive and nonincreasingly ordered 
(including multiplicities) singular values of $A$ with an (either
countably infinite or finite) index set $J$.
The convergence in~\eqref{eq:SVD} is meant with respect to the operator norm.  
The \emph{rank} of $A$ is defined as the cardinality of $J$. 

For $1 \leq p \leq \infty$, the \emph{p-Schatten class} $S_p(H)$ consists 
of all compact operators $A$ on $H$ such that the
norm $\norm{A}_{S_p(H)} := \norm{(\sigma_i(A))_{i \in J}}_{\ell_p}$ is finite.
Here $\norm{(\sigma_i(A))_{i \in J}}_{\ell_p}$ denotes the $\ell_p$ 
sequence space norm of the sequence of singular values.
The spaces $S_p(H)$ are two-sided ideals in $L(H)$. Moreover $ \norm{A} \leq \norm{A}_{S_q(H)} \leq \norm{A}_{S_p(H)}$
holds for $p \leq q$, i.e., $S_p(H) \subseteq S_q(H)$.
For $p=2$, we obtain the Hilbert space of \emph{Hilbert--Schmidt operators}
on $H$ equipped with the inner product $\innerprod{A_1}{A_2}_{\HS{H}} = \tr{A_1^*A_2}$.
For $p=1$, we obtain the Banach algebra
of \emph{trace class operators}. For $p=\infty$, we obtain the Banach algebra of
compact operators equipped with the operator norm $\norm{A} = \norm{A}_{S_\infty(H)}$. The Schatten classes are the completion
of finite-rank operators (i.e., operators in 
$\mspan\{ x \otimes x' \mid x,x' \in H \}$) 
with respect to the corresponding norm. 

We will make frequent use of the fact 
that the tensor product space $H \otimes H$ can be
isometrically identified with the space of Hilbert--Schmidt
operators on $H$, i.e.,
we have $\HS{H} \simeq H \otimes H$. For elements
$x_1, x_1', x_2, x_2' \in H$, we have the relation
$\innerprod{x_2 \otimes x_1}{x_2' \otimes x_1'}_{H \otimes H}
= \innerprod{x_2 \otimes x_1}{x_2' \otimes x_1'}_{\HS{H}}$,
where the tensors are interpreted as rank-one operators 
as described above. This property extends
to $\mspan\{ x \otimes x' \mid x,x' \in H \}$
by linearity and defines a linear isometric isomorphism
between $H \otimes H$ and $\HS{H}$, which can be seen by
considering Hilbert--Schmidt operators in terms of 
their singular value decompositions.

For any topological space $E$, we will write $\sigalg_E = \mathcal{B}(E)$ for its associated
Borel field.
For any collection of sets $\setM$,  $\sigma(\setM)$ denotes
the intersection of all $\sigma$-fields containing $\setM$.
For any $\sigma$-field $\sigalg$ and countable index set $I$, 
we write $\sigalg^{\otimes I}$ as the product $\sigma$-field (i.e., the smallest
$\sigma$-field with respect to which all coordinate projections on $E^I$ are measurable).
Note that when $E$ is Polish (i.e., separable and completely metrizable), we have $\mathcal{B}(E^I) = \mathcal{B}(E)^{\otimes I}$, i.e. the Borel field
on the product space generated by the
product topology and the product of the individual Borel fields are equal.
Put differently, the Borel field operator and the product field operator
are compatible with respect to 
product spaces~\citep[][Proposition 4.1.17]{Dudley}.
Moreover, $E^I$ equipped with the product topology is Polish.

In this paper, we will consider a stochastic process $(X_t)_{t \in \Z}$ on 
a probability space $(\Omega, \sigalg, \Pr)$ with values in the observation space
$(E, \sigalg_E)$, where we assume $E$ to be Polish. 
For a finite number of random variables $\xi_1, \dots, \xi_n$ defined
on $(\Omega, \sigalg, \Pr)$ with values in $E$,
we write $\law(\xi_1, \dots, \xi_n)$ for the \emph{finite-dimensional law},
i.e., the \emph{pushforward measure} on $(E^n, \mathcal{B}(E^n))$ induced
by $\xi_1, \dots, \xi_n$.

\begin{assump}[Stationarity]
    We assume that the process $(X_t)_{t \in \Z}$ is stationary in the sense that
    all finite-dimensional laws are identical, that is
    \begin{equation*}
    \law(X_{t_1}, \dots , X_{t_r}) = \law(X_{t_1 + \lag}, \dots , X_{t_r + \lag})
    \end{equation*}
    for all $t_1, \dots, t_r \in \Z$, $r \in \N$, and time lags $\lag \in \N$. 
\end{assump}

For any separable Banach space $B$, let $L^p(\Omega, \sigalg, \Pr; B)$ denote the
space of strongly $\sigalg-\sigalg_B$ measurable
and Bochner $p$-integrable functions $f \colon \Omega \rightarrow B$
for $ 1 \leq p \leq \infty$~\citep[see for example][]{DiestelUhl1977}. 
In the case of $B = \R$, we simply write $L^p(\Omega, \sigalg, \Pr; \R) = L^p(\Pr)$ for the standard space of real-valued Lebesgue $p$-integrable functions.

\subsection{Reproducing Kernel Hilbert Spaces} \label{sec:RKHS}
We will briefly introduce the concept of reproducing kernel Hilbert spaces. For a detailed discussion
of this topic, we refer the reader to \cite{Berlinet04:RKHS}, \citet{StCh08}
and \citet{SaitohSawano2016}.
To distinguish standard Hilbert spaces from reproducing kernel Hilbert spaces,
we will use the script letter $\inrkhs$ for the latter.

\begin{definition}[Reproducing kernel Hilbert space] \label{def:RKHS}
Let $E$ be a set and $ \mathscr{H} $ a space of functions from $E$ to $\R$. 
Then $ \mathscr{H} $ is called a \emph{reproducing kernel Hilbert space (RKHS)} 
with corresponding inner product $ \innerprod{\cdot}{\cdot}_\mathscr{H} $ if 
there exists function $ k \colon E \times E \to \R $ such that
    \begin{enumerate}[label=(\roman*), itemsep=0ex, topsep=1ex]
        \item $ \innerprod{f}{k(x, \cdot)}_\mathscr{H} = f(x) $ for all $ f \in \mathscr{H} $ (reproducing property), and
        \item $ \mathscr{H} = \overline{\mspan\{k(x, \cdot) \mid x \in E \}}$, where
        the completion is with respect to the RKHS norm.
    \end{enumerate}
    We call $k$ the \emph{reproducing kernel} of $\mathscr{H}$.
\end{definition}

It follows in particular that $ k(x, x^\prime) = \innerprod{k(x, \cdot)}{k(x^\prime,\cdot)}_\mathscr{H} $. The \emph{canonical feature map} $ \varphi \colon E \to \mathscr{H} $ is given by $ \varphi(x) := k(x, \cdot) $. Thus, we obtain $ k(x, x^\prime) = \innerprod{\varphi(x)}{\varphi(x^\prime)}_\mathscr{H} $. Every RKHS has a unique symmetric and positive semi-definite kernel $ k $ with the reproducing property. 
Conversely, every symmetric positive semi-definite kernel $ k $ induces a unique RKHS with $ k $ as its reproducing kernel.
In what follows, we will use the term 
\emph{kernel} synonymously for reproducing kernel/symmetric
positive semi-definite kernel for brevity.

We now impose a few restrictions on the considered RKHS, which we assume
to be fulfilled for the remainder of this paper.

\begin{assump}[Separability] \label{ass:separability}
    The RKHS $\mathscr{H}$ is separable. Note that for a Polish space $E$, 
    the RKHS induced by a continuous kernel $k \colon E \times E \rightarrow \R$ is 
    always separable~\citep[see][Lemma 4.33]{StCh08}.
    For a more general treatment of conditions implying separability, see~\cite{OwhadiScovel2017}.
\end{assump}

\begin{assump}[Measurability] \label{ass:measurability}
    The canonical feature map $\varphi \colon E \rightarrow \inrkhs$ is $\sigalg_E-\sigalg_\inrkhs$ measurable.
    This is the case when $k(x,\cdot) \colon E \rightarrow \R$
    is $\sigalg_E-\sigalg_\R$ measurable for all $x \in E$.
    If this condition holds, then additionally all functions $f \in \inrkhs$ are $\sigalg_E-\sigalg_\R$ measurable 
    and $k \colon E \times E \rightarrow \R$ is $\sigalg_E^{\otimes 2}-\sigalg_\R$ measurable~\citep[see][Lemmas 4.24 and 4.25]{StCh08}.
\end{assump}

\begin{assump}[Existence of second moments] \label{ass:second_moment}
    We have $\varphi(X_0) \in L^2(\Omega, \sigalg, \Pr, \inrkhs)$.
    Note that this is trivially the case whenever $\sup_{x \in E} k(x,x) < \infty$.
\end{assump}

\subsection{Kernel Mean Embeddings and Kernel Covariance Operators}
\label{sec:covariance_operators}

We now introduce kernel mean embeddings and kernel covariance operators,
which are simply the Bochner expectations and covariance operators
of RKHS-embedded random variables.

For a random variable $X$ on $E$ satisfying 
$\varphi(X) \in L^1(\Omega,\sigalg,\Pr;\inrkhs)$,
we call
\begin{equation*}
    \mu_X := \mu_{\law(X)} := \Ex[\varphi(X)] \in \inrkhs
\end{equation*}
the \emph{kernel mean embedding} or simply 
\emph{mean embedding}~\citep{Berlinet04:RKHS,Smola07Hilbert,MFSS16} of $X$.
For every $f \in \inrkhs$, the mean embedding satisfies $\Ex[f(X)] = \innerprod{f}{\mu_X}_\inrkhs$.

\begin{definition}[Kernel (cross-)covariance operator]
For two random variables $X,Y$ on $E$ satisfying
$\varphi(X), \varphi(Y) \in L^2(\Omega,\sigalg,\Pr;\inrkhs)$, we call
the trace class operator $C_{YX} \colon \inrkhs \rightarrow \inrkhs$
defined by
\begin{equation*}
    C_{YX} = \Ex[(\varphi(Y) - \mu_Y) \otimes (\varphi(X) - \mu_X)]
\end{equation*}
the \emph{kernel cross-covariance operator} of
$X$ and $Y$.
We call $C_{XX}$ the \emph{kernel covariance operator}
of $X$.
\end{definition}

For all $f,g \in \inrkhs$, we have $\Cov[f(X),g(Y)] = \innerprod{g}{C_{YX}f}_\inrkhs$ 
as well as $C_{YX} = C_{XY}^*$.
As a consequence, $C_{XX}$ is self-adjoint,
positive semi-definite and trace class.
For additional information about
(cross-)covariance operators of 
Hilbertian random variables, see for example~\citet{Parthasarathy1967},
\citet{Baker70:XCov} and \citet{Baker1973}.

In the literature, the covariance operator is sometimes used as a generalization 
of the uncentered second moment and therefore defined
without centering of the random variables $\varphi(X)$ 
and $\varphi(Y)$~\citep{Prokhorov1956,Parthasarathy1967,Bharucha-Reid1972, Fukumizu13:KBR}.

\begin{definition}[Kernel autocovariance operator]
Let $(X_t)_{t \in \Z}$ be a stationary stochastic process
with values on $E$ such that 
$\varphi(X_0) \in L^2(\Omega,\sigalg,\Pr;\inrkhs)$.
Let $\lag \in \N$.
We call 
\begin{equation*}
	C(\lag) := C_{X_\lag X_0} =  C_{X_{t + \lag} X_{t}}
\end{equation*}
the \emph{kernel autocovariance operator} of the process $(X_t)_{t \in \Z}$ with
respect to the time lag $\lag$.
\end{definition}

\subsection{Product Kernels and Hilbert--Schmidt Operators}
\label{sec:product_kernels}

The tensor product space $\inrkhs \otimes \inrkhs \simeq S_2(\inrkhs)$ 
is itself an RKHS with the corresponding
canonical feature map $	\varphi \otimes \varphi \colon E \times E \rightarrow \inrkhs \otimes \inrkhs$
given by
\begin{equation*}
	\varphi \otimes \varphi\, (x_1,x_2) := \varphi(x_1) \otimes \varphi(x_2).
\end{equation*}
For more details, we refer the reader to \citet[][Lemma 4.6]{StCh08}.
The corresponding kernel of $\inrkhs \otimes \inrkhs$
is the product kernel 
$k \cdot k \colon E^2 \times E^2 \rightarrow \R$.

The estimation of the uncentered kernel autocovariance
operator can therefore be interpreted as the estimation of a kernel mean on
the product RKHS $\inrkhs \otimes \inrkhs$.
In particular, we can write $\cov(\lag)$ as the kernel mean embedding 
of the joint distribution
$\law(X_\lag, X_0)$ on the measurable space $(E \times E, \sigalg_{E} \otimes \sigalg_{E})$
using the product feature map $\varphi \otimes \varphi$.
That is, we have
\begin{equation*}
     \mu_{\law(X_\lag, X_0)} = \Ex[\varphi(X_\lag) \otimes \varphi(X_0)],
\end{equation*} which is exactly the uncentered kernel autocovariance operator of $X$.

Thus, the analysis of the estimation of the uncentered
autocovariance operator covers the problem of estimating 
the kernel mean $\mu_{\law(X_0)}$ of the marginal $\law(X_0)$.
In fact, by considering kernels on the product space $E \times E$ instead
of $E$, we need to account for the challenge that appropriate statistical 
properties of the process 
$(X_t)_{t \in \Z}$ (such as ergodicity, mixing, or decay of correlations) 
have to transfer to the product process $(X_{t + \lag}, X_t)_{t \in \Z}$ on $E \times E$
in order to provide results. 
We may therefore concentrate directly on the estimation
of uncentered autocovariance operators based on $E \times E$ instead on the
estimation of kernel mean embeddings on $E$. We emphasize 
that all of our results directly transfer to
the much simpler case of estimating the kernel mean $\mu_{\law(X_0)}$ 
from dependent data by simply replacing the
product space $E \times E$ with $E$ and the product feature map $\varphi \otimes \varphi$
with $\varphi$.

\begin{assump}[Centered process]
    \label{ass:centered}
    Without loss of generality, we assume that
    the embedded process is centered,
    i.e., $\mu_{\law(X_0)} = \Ex[\varphi(X_0)] = 0$.
    In this case, the centered and the uncentered autocovariance operator
    coincide:
    \begin{equation*}
    C(\lag) = \Ex[\varphi(X_\lag) \otimes \varphi(X_0)].
    \end{equation*}
\end{assump}

\new{%
Whenever Assumption~\ref{ass:centered} is not satisfied,
the centered autocovariance operator
can be computed by replacing the random variables
$\varphi(X_0)$ and $\varphi(X_\lag)$
by their centered counterparts
$\varphi(X_0) - \mu_{\law(X_0)}$ and
$\varphi(X_\lag) - \mu_{\law(X_0)}$, respectively.
It is important to note that unlike $\varphi(\cdot)$,
the expression $\varphi(\cdot) - \mu_{\law(X_0)}$ is technically not the canonical
feature map of $\inrkhs$ and should be treated with caution.
In particular, it does not satisfy the reproducing property.
It does however
satisfy all conditions required to formally define its mean embedding
and the corresponding covariance operators.
Furthermore, in practical applications, the mean embedding
$\mu_{\law(X_0)}$ is typically not known and replaced by an 
empirical estimate.
This centering and its consequences
regarding the empirical estimation of kernel covariance operators
are investigated in detail by \citet{Blanchard2007}.
}

In what follows, we will repeatedly use the shorthand
\begin{align*}
C_n(\lag) &:= \frac{1}{n} \sum_{t = 1}^n \varphi(X_{t + \lag}) \otimes \varphi(X_t)
\end{align*}
for the empirical estimate of $C(\lag)$ based on $n + \lag$ consecutive time steps
of the process $(X_t)_{t \in \Z}$.

\section{Strong Law of Large Numbers}
\label{sec:SLLN}

We now address the strong law of large numbers for the estimator $C_n(\lag)$.
To this end, we briefly introduce the concept of 
measure-preserving dynamical systems
and ergodicity. For details, the reader may refer for example 
to~\cite{Petersen1983}. 
\new{
It is well-known that every stationary process
can be expressed in terms of a measure-preserving dynamical system
when the underlying probability space is chosen 
accordingly~(see for example \citealp{Doob1953}, Chapter X).
This allows to study stationary processes with tools from
ergodic theory. 
We will briefly describe this viewpoint below.
}

It is possible to assume without loss of generality that the 
underlying probability space  $(\Omega, \sigalg, \Pr)$
describing $(X_t)_{t \in \Z}$ is the \emph{canonical probability space}, i.e.,
$\Omega := E^\Z$ and $\sigalg := \sigalg_E^{\otimes \Z}$.
In this case, we can simply express the process $(X_t)_{t \in \Z}$
as the family of \emph{coordinate projections} on $\Omega$: 
for $\omega = (\omega_t)_{t \in \Z} \in \Omega$,
we can write 
\begin{equation*}
    X_t(\omega):= \omega_t = X_0(T^t \omega) \quad \textnormal{for all } t \in \Z,
\end{equation*}
where $T$ is the \emph{left-shift operator} 
on $\Omega$ defined by $(T\omega)_i = \omega_{i+1}$ for all $i \in \Z$. 
Note that by stationarity of
$(X_t)_{t \in \Z}$, the shift $T$ is measure preserving in the sense that
$ \Pr[T^{-1}M]= \Pr[M]$ for all $M \in \sigalg^{\otimes\Z}_E$. 
We call $(X_t)_{\in \Z}$ \emph{ergodic}
whenever $T$ is \emph{ergodic in the measure theoretical sense},
i.e., for all sets $M \in \sigalg_E^{\otimes \Z}$, we have that 
the condition $ T^{-1}M = M$ implies either
$\Pr[M] = 0$ or $\Pr[M]=1$.

We show that for any fixed time lag $\lag \in \N$, the kernel 
autocovariance operator $\cov(\lag)$ can be
estimated almost surely from realizations of $(X_t)_{t \in \Z}$
whenever the process is ergodic.
This comes as a natural consequence of the following 
generalized version of Birkhoff's ergodic theorem.

\begin{theorem}[\citealp{BeckSchwarz1957}, Theorem 6] \label{thm:beck-schwartz}
    Let $B$ be a reflexive Banach space and $T$ an ergodic measure-preserving transformation
    on $(\Omega, \sigalg, \Pr)$. Then for each $f \in L^1(\Omega, \sigalg, \Pr; B)$,
    \begin{equation*}
    \lim_{n \to \infty} \frac{1}{n} \sum_{i = 1}^{n} f(T^i \omega ) = \Ex[f],
    \end{equation*}
    where the convergence holds $\Pr$-a.e.\ with respect to $\norm{\cdot}_B$.
\end{theorem}

We can directly apply this result to obtain almost sure convergence
of the empirical estimate of $\cov(\lag)$ in the case that $(X_t)_{t \in \Z}$ is ergodic.

\begin{corollary}[Strong consistency] \label{cor:SLLN}
    Let $(X_t)_{t \in \Z}$ be a stationary and ergodic process defined on $(\Omega, \sigalg, \Pr)$ with
    values in a Polish space $E$. Then 
     \begin{equation*}
    \lim_{n \to \infty} C_n(\lag) = \cov(\lag),
    \end{equation*}
    where the convergence is $\Pr$-a.e.\ with respect to $\norm{\cdot}_{\HS{\inrkhs}}$.
\end{corollary}
\begin{proof}
    The time-lagged product process $(X_t, X_{t + \lag})_{t \in \Z}$
    on $E \times E$ can be expressed via the projection tuple 
    $(X_t, X_{t + \lag })(\omega) = (X_0, X_\lag)(T^t \omega)$.
    By construction, $(X_0, X_\lag)$ is $\Pr - \sigalg_{E} \otimes \sigalg_{E}$ measurable.
    Note that because of 
    Assumption~\ref{ass:measurability} and Assumption~\ref{ass:second_moment} 
    the product feature map 
    $\varphi \otimes \varphi$ given by $(x,y) \mapsto \varphi(y) \otimes \varphi(x)$ is an element of
    $L^1(E \times E, \sigalg_{E} \otimes \sigalg_{E}, \law(X_0, X_\lag); \HS{\inrkhs})$, 
    where $\HS{\inrkhs}$ is clearly reflexive.
    Therefore, it holds that the composition 
    \begin{equation*}
    \varphi \otimes \varphi \circ (X_0, X_\lag) \colon \Omega \rightarrow \HS{\inrkhs}
    \end{equation*}
    given by $\omega \mapsto (X_0, X_\lag)(\omega) \mapsto \varphi(X_\lag) \otimes \varphi(X_0)(\omega)$
    is an element of $L^1(\Omega, \sigalg, \Pr; \HS{\inrkhs})$.
    
    The statement follows immediately from the fact that 
    we choose $\varphi \otimes \varphi \circ (X_0, X_\lag)$ as the observable $f$ in
    Theorem~\ref{thm:beck-schwartz} and obtain
    \begin{equation*}
    \lim_{n \to \infty} \frac{1}{n} \sum_{t = 1}^{n} \varphi \otimes \varphi \circ (X_0, X_\lag) \circ T^t =
    \lim_{n \to \infty} \frac{1}{n} \sum_{t = 1}^{n} \varphi(X_{t + \lag}) \otimes \varphi(X_t) = \cov(\lag),
    \end{equation*}
    where the convergence is $\Pr$-a.e.\ in $\HS{\inrkhs}$.
\end{proof}

\begin{remark}[Convergence in Schatten norms]
    Corollary~\ref{cor:SLLN} also yields $\Pr$-a.e.\ convergence $C_n(\lag) \to C(\lag)$ 
    in $S_p(\inrkhs)$ for all $p \geq 2$,
    \new{%
     as it holds that
    $\norm{\cdot}_{S_p(\inrkhs)} \leq \norm{\cdot}_{S_2(\inrkhs)}$.}
    Note that $S_1(\inrkhs)$ is reflexive if and only if $\inrkhs$ is 
    finite-dimensional~\citep[see for example][Theorem 3.2]{Simon2005:Trace}.
    However, in the finite-dimensional case, all Schatten classes coincide and the question
    for convergence in Schatten norms becomes trivial.
    In the general case, it is not clear whether the 
    reflexivity assumption in Theorem~\ref{thm:beck-schwartz}
    is not only sufficient but also necessary for a convergence to hold.
    To the best of our knowledge, no stronger generalization results
    of Birkhoff's ergodic theorem for Banach-valued random variables exist.
\end{remark}

\section{Asymptotic Error Behavior}
\label{sec:asymptotic_error}

\new{%
In this section, we apply
limit theorems for weakly dependent Hilbertian random variables
to the estimation error $\cov[n](\lag)- \cov(\lag)$ and investigate its asymptotic
behaviour.
}

\subsection{$\alpha$-Mixing}
We first recall the basic notions of
$\alpha$-mixing in 
statistics~\citep[see for example][]{Bradley2005}. 

\begin{definition}[$\alpha$-mixing coefficient]
\label{defi:mixing}
For $\sigma$-fields $\sigalg_1$, $\sigalg_2 \subseteq \sigalg$, we define
\begin{equation*}
	\alpha(\sigalg_1,\sigalg_2) := \sup_{A \in \sigalg_1, B \in \sigalg_2} 
	\abs{\Pr[A \cap B] - \Pr[A]\Pr[B]}.
\end{equation*}
For a process $(X_t)_{t \in \Z}$, we furthermore define
\begin{equation*}
    \label{eq:mixing}
  	\alpha(n) := \alpha((X_t)_{t \in \Z}, n) :=
    \begin{cases}
        \new{%
        \sup\limits_{j \in \Z} \, 
    \alpha (\sigalg_{-\infty}^{j},\sigalg_{j+n}^{\infty})}, \quad &\new{n \geq 1,}
        \\
    \new{1/4}, \quad  &\new{n < 1, }
    \end{cases}
\end{equation*}
where $\sigalg_l^m := \sigma(X_t, l \leq t \leq m)$ denotes
the $\sigma$-field generated by the process $(X_t)_{t \in \Z}$ for time horizons
$- \infty \leq l \leq m \leq \infty $.
\end{definition}

The process is called $\alpha$-\emph{mixing} or \emph{strongly mixing},
when $\alpha(n) \to 0$ as $n \to \infty$.
In this case, the convergence rate of $\alpha(n)$ 
is called the \emph{mixing rate} of the associated process.
In this paper, we will not focus on the various 
alternative strong mixing coefficients which are frequently used
in statistics~\citep{Doukhan94,Bradley2005, Rio2017Asymptotic}, 
since $\alpha$-mixing is 
the weakest concept among the strong mixing coefficients and covers a wide range
of processes in practice.
\new{%
The coefficient 
$\alpha(n)$ is typically defined only for $n \geq 1$
in the literature. However, since it
always holds that $0 \leq \alpha(\sigalg_1, \sigalg_2) \leq 1/4$,
our definition above extends $\alpha(n)$  trivially to $n \in \Z$.
This allows us to formulate several results
in a more convenient way in what follows.
}

\begin{remark}[Terminology]
    The concept of \emph{strong mixing coefficients} is typically much stronger
	than the \emph{strong mixing} considered in 
	ergodic theory~\citep{Petersen1983}. Also note that we define
    the $\alpha$-mixing coefficient for stationary processes.
	It can also be defined for nonstationary processes,
	while mixing in the ergodic theoretical sense
	typically arises from dynamical systems induced by
    measure-preserving transformations and is therefore primarily 
    used in the context of stationary stochastic processes.
\end{remark}

\begin{example}[Mixing processes] \label{ex:mixing_processes}
    A wide range of mixing processes are given by \cite{Doukhan94} and
    \cite{Bradley2005}. We list some important examples here.
    \begin{enumerate}
        \item Irreducible and aperiodic stationary Markov processes 
        on $E \subseteq \R$ are $\alpha$-mixing 
        (in fact, a stronger mixing property called $\beta$\emph{-mixing} or 
        \emph{absolute regularity} holds, see for 
        example~\citealp[Corollary 3.6]{Bradley2005}).
        \item Stationary Markov processes satisfying
        \emph{geometric ergodicity}~\citep[for details see][Chapter 15]{MeTw12} 
        are $\alpha$-mixing with 
        $\alpha(n) = O(\exp(-cn))$ for some $c>0$
         \citep[Theorem 3.7]{Bradley2005}.
        \new{
        \item We consider a stochastic dynamical
        system $(X_t)_{t \in \N_0}$ also
        known as the \emph{nonlinear state space model} \citep[Chapter 7]{MeTw12}
        given by the recursion 
        \begin{equation*}
        	\label{eq:nonlinear_state_space_model}
        	X_t = h_t( X_{t-1}, Z_t),  \quad t \geq 1,
        \end{equation*} 
        where $h_t \colon E \rightarrow E$ are measurable
        functions and $Z_t$ are i.i.d.\ random variables on $E$
        which are independent of $X_0$.
        The process $(X_t)_{t \in \N_0}$ is a 
        Markov process~\cite[see for example][\rev{Proposition 8.6}]{Kallenberg}.
        Therefore 1. and 2. apply in this case.
        General conditions under which this system is geometrically ergodic 
        (i.e., geometrically $\alpha$-mixing in the sense of 2.) are 
        presented by~\citet[Section 2.4]{Doukhan94}.
        As a very basic example,
        the so-called \emph{simple linear model} on $E \subseteq \R$ given by
        \begin{equation*}
        	X_t = a X_{t-1} + Z_t, \quad a \in \R, \, t \geq 1,
        \end{equation*} is geometrically ergodic  
        if \rev{$\abs{a} < 1$} and
        the random variables $Z_t$ 
        \rev{are integrable and}
        admit an everywhere 
        positive density on $E$ \citep[Section 15.5.2]{MeTw12}.
        }
        \item Under some requirements, commonly used linear 
        and nonlinear process models on finite-dimensional
        vector spaces including AR, ARMA, ARCH, and GARCH are $\alpha$-mixing with 
        $\alpha(n) = O(\exp(-cn))$ for some $c>0$, see~\citet[Section 2.4]{Doukhan94}
        and~\citet[Section 2.6.1]{FanYao03}.       
    \end{enumerate}
    
\end{example}

We make use of the following classical lemma which we prove for completeness. 
It ensures that the mixing rates of
measurable transformations of a stationary process
process are at least as fast as the mixing rates of the original process.

\begin{lemma}[Mixing coefficients of transformed processes] 
    \label{lem:mixing_transformation}
 	Let $(X_t)_{t \in \Z}$ 
 	defined on $(\Omega, \sigalg, \Pr)$ be a process
    with values in the Polish space $(E, \sigalg_E)$
    equipped with its Borel field. 
    Let $(F, \sigalg_F)$ be another Polish space equipped with its Borel field.
 	Let $\eta \in \N$ and $h \colon E^{\eta+1} \rightarrow F$ 
 	be a $\sigalg_E^{\otimes \lag + 1}-\sigalg_F$ measurable transformation.
 	Then for the $F$-valued process $(H_t)_{t \in \Z}$ given by
  	\begin{equation*}
 	H_t := h(X_{t}, \dots, X_{t + \eta}), \quad t \in \Z,
 \end{equation*}
 	we have 
 	\begin{equation}\label{eq:mixing_transformation}
        \new{%
 		\alpha((H_t)_{t \in \Z}, n) \leq \alpha((X_t)_{t \in \Z}, n-\eta)
        }
 	\end{equation}
 	for all $n \in \Z$.
 	In particular, if $(X_t)_{t \in \Z}$ is $\alpha$-mixing,
 	then $(H_t)_{t \in \Z}$ is $\alpha$-mixing with the
 	same mixing rate as $(X_t)_{t \in \Z}$ or faster.
 \end{lemma}

 \begin{proof} 
     \new{%
    It is sufficient to consider the case
    $n > \lag$, as \eqref{eq:mixing_transformation} 
    is trivial whenever $n \leq \lag$
    by our definition of the $\alpha$-mixing coefficients.
    }
 	Let $\mathcal{H}_l^m := \sigma(H_t, l \leq t \leq m) \subseteq \sigalg$ 
 	be the $\sigma$-field generated by $(H_t)_{t \in \Z}$
 	for time horizons $l$ and $m$.	
 	By construction, for all $j \in \Z$ we have 
 	$\mathcal{H}_{- \infty}^{j} \subseteq \sigalg_{-\infty}^{j+\eta}$ as well as
 	$\mathcal{H}_{j}^{\infty} \subseteq \sigalg_{j}^\infty$ 
 	for all $j \in \Z$. 
 	Therefore, we have
 	\begin{equation*}
 		\alpha((H_t)_{t \in \Z}, n) 
        =  \sup_{j \in \Z} \alpha (\mathcal{H}_{-\infty}^{j},\mathcal{H}_{j+n}^{\infty})
 		\leq  \sup_{j \in \Z} \alpha (\sigalg_{-\infty}^{j+\eta},\sigalg_{j+n}^{\infty}) 
        = 	\alpha((X_t)_{t \in \Z}, n-\eta),
 	\end{equation*}
    proving the claim.
 \end{proof}

\subsection{Limit Theorems}

We now investigate the asymptotic statistical behavior 
of the estimator $\cov[n](\lag)$ under the assumption that $(X_t)_{t \in \Z}$ is strongly mixing.
For brevity, we introduce the shorthand notation
\begin{equation} \label{eq:process_shorthand}
	\xi_{t} := \big( \varphi(X_{t + \lag}) \otimes \varphi(X_{t}) \big) - C(\lag)
\end{equation}
for $t \in \Z$ and $\lag \in \N$. 
The process $(\xi_t)_{t \in \Z}$
is stationary and centered with values in $\HS{\inrkhs}$.
The estimation error can now be expressed as
$\cov[n](\lag) - \cov(\lag) = \tfrac{1}{n} \sum_{t =1 }^n \xi_t $.

\begin{lemma}[Mixing process in product RKHS] 
    \label{lem:mixing_transformation2}
\new{%
Let $\lag \in \N$ be the time lag used to define the 
$S_2(\inrkhs)$-valued process $(\xi_t)_{t \in \Z}$ given in 
by \eqref{eq:process_shorthand}. 
}
Then the $\alpha$-mixing coefficients of
of $(\xi_t)_{t \in \Z}$
satisfy
\begin{equation}
	\label{eq:mixing_transformation2}
    \new{%
	\alpha((\xi_t)_{t \in \Z}, n) \leq \alpha((X_t)_{t \in \Z}, n-\eta)
    }
    \quad \textnormal{for all } n \in \Z.
\end{equation}
In particular, if the process $(X_t)_{t \in \Z}$ is $\alpha$-mixing,
then $(\xi_t)_{t \in \Z}$ is $\alpha$-mixing with the
same rate as $(X_t)_{t \in \Z}$ or faster.
\end{lemma}

\begin{proof}
    It is sufficient to note that
    $(\xi_t)_{t \in \Z}$
    is obtained from $(X_t, \dots, X_{t + \lag})_{t \in \Z}$ via
    a measurable transformation.
    The assertion follows from Lemma~\ref{lem:mixing_transformation}.
\end{proof}

We will make use of the fact that whenever
the kernel $k$ is bounded, the process $(\xi_t)_{t \in \Z}$
is almost surely bounded. In particular, if
 $\sup_{x \in E} k(x,x) = c < \infty$,
then for all $t \in \Z$, we have
\begin{align}
	\label{eq:bounded-process}
	\begin{split}
		\norm{\xi_t}_{L^{\infty}( \Omega, \sigalg, \Pr;  \HS{\inrkhs}) } 
		&\leq \esssup_{\omega \in \Omega} \norm{ \varphi(X_{t + \lag}) 
            \otimes \varphi(X_{t} ) }_{\HS{\inrkhs}}  
		+ \norm{ C(\lag) }_{\HS{\inrkhs}}  \\
		&\leq \sup_{x \in E}  \norm{\varphi(x)}^2_\inrkhs 
        + \Ex \left[ \norm{ \varphi( X_{t + \lag} ) 
            \otimes \varphi( X_t )  }_{\HS{\inrkhs} }  \right] \\
		& \leq 2 \sup_{x \in E} \norm{\varphi(x)}^2_\inrkhs 
		= 2 \sup_{x \in E} k(x, x)
		= 2c.
	\end{split}
\end{align}

Several properties of the estimation error $\cov[n](\lag) - \cov(\lag)$ can
be proven by applying results from the asymptotic theory of weakly dependent 
Hilbertian processes to $(\xi_t)_{t \in \Z}$. We begin with
one of the strongest results of this type 
which is an approximation of 
the rescaled estimation error $n(\cov[n](\lag) - \cov(\lag))$ by a Gaussian process.

\new{%
To this end, we briefly recall the definition of a Gaussian measure on a 
Hilbert space \citep[see for example][]{BogachevGaussianMeasures},
which generalizes the concept of the finite-dimensional normal distribution.
A probability measure $\nu$ on a Hilbert space $H$ equipped 
with its Borel field is called
a \emph{Gaussian measure} if its characteristic function $\tilde{\nu}: H \rightarrow \C$
defined by $\tilde{\nu}(x) := \int_H \exp( i \innerprod{x}{y}_H ) \, \dd \nu(y)$
is of the form
\begin{equation}
	\label{eq:gaussian_measure}
	\tilde{\nu}(x) = \exp \left( \innerprod{\mu}{x}_H - \frac{1}{2} \innerprod{x}{T x}_H \right)
\end{equation}
for some $\mu \in H$ (the \emph{mean} of $\nu$) and positive-semidefinite
self-adjoint operator $T \in S_1(H)$
(the \emph{covariance operator} of $\nu$, see
\citealp[Theorem 2.3.1]{BogachevGaussianMeasures}).
If $\nu$ satisfies \eqref{eq:gaussian_measure}, we also write $\nu = \mathcal{N}(\mu,T)$.
We note that the covariance operator $T$ should not be confused with the kernel
autocovariance operator $\cov(\lag)$, as $T$ describes a
limiting covariance structure
of the sequence of 
laws associated with the empirical estimates of $\cov(\lag)$ in what follows.}\\
\rev{We now introduce the shorthand notation
    $L(n) := \max( \log n , 1)$ for $n \in \N$.}

\begin{theorem}[Almost sure invariance principle] \label{thm:ASIP}
    Let $(X_t)_{t \in \Z}$ be stationary and 
    $\alpha$-mixing with coefficients $(\alpha(t))_{t \in \Z}$ 
    such that
    $\sum_{t = 1}^\infty \alpha(t) < \infty$.
    Furthermore, let $\sup_{x \in E} k(x,x) < \infty$. 
    \new{%
    Let $\lag \in \N$ be the time lag used to define the 
    $S_2(\inrkhs)$-valued process $(\xi_t)_{t \in \Z}$ given in 
    by \eqref{eq:process_shorthand}. 
    }
    Then the  linear operator 
    $T_\lag \colon \HS{\inrkhs} \rightarrow \HS{\inrkhs} $
    defined by
    \begin{equation} \label{eq:asymptotic_cov_operator}
        T_\lag 
        := \Ex[ \xi_0 \otimes \xi_0 ] 
        + \sum_{t = 1}^\infty \Ex[ \xi_0 \otimes \xi_t ] +
         \sum_{t = 1}^\infty \Ex[ \xi_t \otimes \xi_0 ] 
    \end{equation} is trace class.
    Furthermore, there exists a Gaussian measure 
    $\mathcal{N}(0,T_\lag)$ on $\HS{\inrkhs}$
    and a sequence of i.i.d. $\HS{\inrkhs}$-valued Gaussian random variables 
    $(Z_t)_{t \in \Z} \sim \mathcal{N}(0,T_\lag)$
    defined on $(\Omega, \sigalg, \Pr)$ such that  we have $\Pr$-a.e.
    \begin{equation*}
        \norm{n (\cov[n](\lag) - \cov(\lag)) - \sum_{t = 1}^n Z_t }_{\HS{\inrkhs}} 
        = o\left(\sqrt{n L(L(n))} \right).
    \end{equation*} 
\end{theorem}
\begin{proof}
    The assumptions ensure that 
    $(\xi_t)_{t \in \Z}$ is $\Pr$-a.e.\ bounded
    by \eqref{eq:bounded-process} and has 
    summable mixing coefficients by \eqref{eq:mixing_transformation2}.
    We can directly apply the almost sure invariance principle 
    from~\citet[Corollary 1]{DedeckerMerlevede2010:ASIP}
    to $(\xi_t)_{t \in \Z}$, which yields the assertion.
    \new{%
    We emphasize
    that the $\Pr$-a.e.\ boundedness of $(\xi_t)_{t \in \Z}$ and 
    the summability of the mixing coefficients ensure that
    this result can be applied, as explicitly
    mentioned by the authors (see also \citealp{Merlevede2008:LIL}, Remark 3).
    }
\end{proof}

A strongly related statement is a 
standard central limit theorem for weakly dependent sequences
which ensures asymptotic normality in the space $\HS{\inrkhs}$.

\begin{theorem}[Central limit theorem] \label{thm:CLT}
    \new{Let $\lag \in \N$ be a fixed time lag.}
    Under the assumptions of Theorem~\ref{thm:ASIP}, 
    the laws of the sequence $\sqrt{n} (C_n(\lag) - C(\lag))$
    converge weakly to a Gaussian measure $\mathcal{N}(0,T_\lag)$ 
    on $\HS{\inrkhs}$ 
    with covariance operator $T_\lag$ 
    defined by~\eqref{eq:asymptotic_cov_operator}.
\end{theorem}

\begin{proof}
    By our previous analysis 
    \new{%
        and the argumentation
        in the proof of Theorem~\ref{thm:ASIP},
    }
    the process $(\xi_t)_{t \in \Z}$
    satisfies all assumptions of the central limit theorem by
    \citet[Corollary 1]{MerlevedeEtAl97}. The above assertions
    follow directly.
\end{proof}

The next result is a compact law of the iterated logarithm.
It ensures that an appropriately rescaled version
of the estimation error approximates a compact limiting set almost surely.
Additionally, it characterizes this set as the accumulation points of 
the estimation error sequence and
gives a norm bound in $\HS{\inrkhs}$ depending on the mixing rate.
\new{%
We obtain our result by applying an
infinite-dimensional generalization of the classical law of 
the iterated logarithm for real-valued random variables.
The original law of the iterated logarithm 
plays an important role in sequential hypothesis testing
\citep{Robbins70, Siegmund85}, leading to potential applications
in reinforcement learning 
as noted by \cite{Kaufman16}.}
\\

Let $\acc(x_n) \subseteq X$ denote the set of all 
\emph{accumulation points} of a sequence $(x_n)_{n \in \N}$ in a 
metric space $(X, d)$
\new{%
    and $\dist(x, A) := \inf \{ d(x,y) \, | \, y \in A \}$
    be the \emph{distance} between a point $x \in X$
    and a set $A \subseteq X$.
}

\begin{theorem}[Compact law of the iterated logarithm] \label{thm:clil}
    Let the process $(X_t)_{t \in \Z}$ be stationary and 
    $\alpha$-mixing with coefficients $(\alpha(t))_{t \in \Z}$ 
    such that
    $\sum_{t = 1}^\infty \alpha(t) < \infty$.
    Furthermore, let 
    $\sup_{x \in E} k(x,x)\rev{=c}< \infty$. 
    Then \new{for every time lag $\lag \in \N$} there exists a 
    compact, convex and symmetric set $K_\lag \subseteq \HS{\inrkhs}$,
    such that $\Pr$-a.e.
    \begin{equation} \label{eq:LIL-lim}
        \lim_{n \to \infty} 
        \dist\left(\frac{\sqrt{n}(\cov[n](\lag) 
            - \cov(\lag))}{\sqrt{2L(L(n))}} , K_\lag \right)  = 0
    \end{equation} as well as $\Pr$-a.e.
    \begin{equation} \label{eq:LIL-acc}
        \acc \left(
        \frac{\sqrt{n}(\cov[n](\lag) - \cov(\lag))}{\sqrt{2L(L(n))}} 
    \right) 
        = K_\lag.
    \end{equation}
    Moreover, whenever $\sum_{t = 1}^\infty \alpha(t - \lag) = M < \infty$,
    we have
    \begin{equation} \label{eq:LIL-bound}
        \sup_{A \in K_\lag} \norm{A}_{\HS{\inrkhs}} 
        = (4c^2 + 32\, c^2 M)^{1/2}.
    \end{equation}
\end{theorem}

For the proof of Theorem~\ref{thm:clil}, see Appendix~\ref{app:proofs}.
\rev{%
We note that from \eqref{eq:LIL-acc}, we can deduce
\begin{equation*}
    \limsup_{n \to \infty}
    \frac{
        \sqrt{n}
        \norm{\cov[n](\lag) - \cov(\lag)}_{\HS{\inrkhs}}
    }
    {\sqrt{2L(L(n))}} 
    = 
    \sup_{A \in K_\lag} \norm{A}_{\HS{\inrkhs}},
\end{equation*}
which gives us the rate of
$\Pr$-a.e.\ convergence of the estimation error.
\begin{corollary}[Rate of convergence]
    \label{cor:convergence_rate}
    Under the assumptions of Theorem~\ref{thm:clil},
    we have
    \begin{equation*}
        \norm{\cov[n](\lag) - \cov(\lag)}_{\HS{\inrkhs}} = 
        O \left(
        \frac{\sqrt{2L(L(n))}}%
        {\sqrt{n}}
        \right)
        \quad \Pr\textnormal{-a.e.}
    \end{equation*}
    for every time lag $\lag \in \N$.
\end{corollary}
In particular, Corollary~\ref{cor:convergence_rate}
shows that boundedness of the kernel and summability of the
mixing coefficients are sufficient to obtain
the rate of convergence which is known
to hold in the strong law of large numbers 
for independent random variables as given by the classical law
of the iterated logarithm
(see for example \citealp{Acosta1983}).
}

    \begin{remark}[Optimality of mixing rate assumptions]
    \new{
    The results in this section require that the 
    $\alpha$-mixing coefficients of $(X_t)_{t \in \Z}$
    are summable. We now briefly address the question whether
    similar asymptotic statements can be derived under less strict assumptions.
    \citet{MerlevedeEtAl97}, \citet{Merlevede2008:LIL} 
    and \citet{DedeckerMerlevede2010:ASIP}
    use more general and much more technical quantile conditions than the summability
    of the mixing coefficients in order to prove the asymptotic results which 
    we apply here.
    These quantile conditions are known to be
    necessary for a central limit theorem to hold for real-valued processes. 
    We refer the reader to~\citet[Section 4]{Doukhan1994:CLT} for additional information. 
    For bounded random variables however, the summability of the 
    mixing coefficients is equivalent
    to the mentioned quantile conditions.
    This is investigated by~\citet[Application 1]{Rio1995}.
    A similar argument
    can be used for the law of the iterated logarithm~\cite[see also][]{Rio1995}.
    }
\end{remark}

\section{Concentration Bounds}
\label{sec:concentration_bounds}

In addition to the previous asymptotic results,
concentration properties for the estimation error can be derived
by using concentration properties of mixing Hilbertian processes.
\new{%
We now introduce the covariance operator of the
$\HS{\inrkhs}$-valued random variable
$\xi_t$ defined in~\eqref{eq:process_shorthand},
which we will denote by $\Gamma_{\lag} \in S_1(\HS{\inrkhs})$.
The eigendecomposition of
$\Gamma_{\lag}$ allows to quantify the second moment
of the empirical kernel autocovariance operators $C_n(\eta)$,
which is needed to derive concentration bounds.
We emphasize that 
the operator $\Gamma_{\lag}$ should
not be confused with the operator $T_\lag$ introduced 
in \eqref{eq:asymptotic_cov_operator}, which describes the 
\emph{asymptotic covariance} of $\sqrt{n}(C_n(\lag)- C(\lag))$,
while $\Gamma_{\lag}$ describes the variance of the marginal
$\law(\xi_0)$
of the stationary process $(\xi_{t})_{t \in \Z}$.
}

\begin{theorem}[Error bound] \label{thm:finite_sample_bound}
    Let $(X_t)_{t \in \Z}$ be stationary
    with $\alpha$-mixing coefficients $(\alpha(t))_{t \in \Z}$
    Let furthermore $\sup_{x \in E} k(x,x) = c < \infty$.
    Then for every time lag $\lag \in \N$,
    $\epsilon > 0$, $\nu \in \N$, $n \geq 2$ as well as
    $q \in \{1, \dots, \floor*{n/2} \}$ and
    $\delta \in (0,1)$, we have
    \begin{align*}
        \Pr{ \left[ \norm{C_{n}(\lag) - C(\lag)}_{\HS{\inrkhs}}  > \epsilon \right] }
        &\leq 4 \nu \exp\left(- \frac{(1-\delta) \epsilon^2 q}{32 \nu c^2} \right) \\
        &+ 22\nu q \left( 1 + \frac{8c \new{\sqrt{\nu}}}%
                {\epsilon(1-\delta)^{1/2}}  
        \right)^{1/2} \alpha(\floor*{n/2q} \new{- \lag}) 
        + \frac{1}{\delta \epsilon^2} \sum_{j > \nu} \lambda_j,
    \end{align*}
    where the nonnegative real numbers $(\lambda_j)_{j \geq 1}$ 
    are the nonincreasingly ordered eigenvalues of the covariance operator 
    \begin{equation*}
    	\Gamma_{\lag} \colon \HS{\inrkhs} \rightarrow \HS{\inrkhs}
	\end{equation*}
    defined by
    \begin{equation} \label{eq:covcov_operator}
   \Gamma_{\lag} := 
   \Ex \left[ 
   \Big( \big( \varphi(X_{\lag}) \otimes \varphi(X_{0}) \big) - C(\lag) \Big)
   \otimes
   \Big( \big( \varphi(X_{\lag}) \otimes \varphi(X_{0}) \big) - C(\lag) \Big)
   \right].
    \end{equation}
\end{theorem}

\begin{proof}
    As previously noted, the process $(\xi_t)_{t \in \Z}$ 
    defined by~\eqref{eq:process_shorthand}
    \new{ with the time lag $\eta \in \N$}
    is stationary, centered and almost surely bounded by $2c$ in the
    norm of $\HS{\inrkhs}$.
    Moreover, its
    $\alpha$-mixing coefficients satisfy the bound
    \eqref{eq:mixing_transformation2}.
    We can therefore apply the concentration 
    bound given by~\citet[][Theorem 2.12]{Bosq2000}
    to the process $(\xi_t)_{t \in \Z}$,
    which yields the assertion.
\end{proof}

    The above bound requires an optimal trade-off between
    $\nu$, $q$, and $\delta$.
    In particular, knowledge about the decay of the eigenvalues 
    $(\lambda_j)_{j \geq 1}$ of $\Gamma_{\lag}$
    allows to choose $\nu$, $q$ and $\delta$
    \rev{%
    such a way that the bound can be simplified
    as demonstrated by \citet[Corollary 2.4]{Bosq2000}.
    We also note that 
    $\sum_{j > \nu} \lambda_j < \infty$ for every $\nu \in \N$,
    since $\Gamma_{\lag}$ is trace class.
    }

\section{Consistency of Weakly Dependent Kernel PCA}
\label{sec:kernel_pca}

By considering the kernel covariance operator $\cov:= \cov(0)$,
we can easily obtain consistency results 
for kernel PCA~\citep{Scholkopf98:KPCA} for the case that 
the data is dependent.
It is well 
known that kernel PCA approximates the spectral decomposition of $C$
\citep[see for example][]{Blanchard2007}, as we will 
briefly explain in Section~\ref{sec:pca_interpretation}.
Consistency results for kernel PCA from independent data
have been obtained by considering the spectral perturbation
of covariance operators of Hilbertian random variables
~\citep{Mas2003, Blanchard2007, Mas2015, Koltchinskii2016, Koltchinskii2017, Reiss2020}.
Various approaches exist in this context and we do not
aim to provide a full overview here. Instead, we will show how
our previous results lead to some elementary consistency statements
for dependent data.
By applying techniques from the previously mentioned literature,
these results may be refined and extended accordingly.

We note that convergence in measure and weak convergence of 
standard linear Hilbertian PCA for
$L^2([0,1])$-valued stochastic processes
was previously investigated by \citet{KokoszkaReimherr2013}
under the assumption of $L^4$-$m$ approximability.

\subsection{Notation}
\label{sec:pca_notation}

For a compact self-adjoint positive-semidefinite operator $\cov$ on $\inrkhs$, let
$(\lambda_i(\cov))_{i \in I}$ denote the nonzero eigenvalues of $\cov$ ordered 
nonincreasingly repeated with their multiplicities for the index set 
$I= \{1,2, \dots\}$.
Then $\cov$ admits the spectral decomposition
\begin{equation}
    \label{eq:cov_spectral_decomposition}
\cov = \sum_{i \in I} \lambda_i(\cov)\, v_i \otimes v_i
\end{equation}
where the $v_i$ are the orthonormal eigenfunctions of $\cov$.
In addition, let $(\mu_j(\cov))_{j \in J}$ denote the \emph{distinct}
eigenvalues of $C$ for $J = \{1,2, \dots\}$ with
$\triangle_j(\cov) := \{i \in I \mid \lambda_i(\cov) = \mu_j\}$
as well as the multiplicity $m_j(\cov) := \abs{\triangle_j(\cov)}$.
Note that $\cov$ can also be written as
\begin{equation*}
\cov = \sum_{j \in J} \mu_j(\cov) P_j(\cov),
\end{equation*}
were $P_j(\cov)$ is the orthogonal spectral projector onto the eigenspace
corresponding to $\mu_j(\cov)$ and the convergence is with respect to the operator norm.
We will additionally consider the \emph{spectral gap} 
\begin{equation*}
g_j(\cov) := 
\begin{cases}
\mu_{1}(\cov)- \mu_{2}(\cov), & j = 1, \\
\min \{\mu_{j-1}(\cov) - \mu_{j}(\cov), \, \mu_{j}(\cov)- \mu_{j+1}(\cov) \}, & j \geq 2.
\end{cases}  
\end{equation*}
Note that $g_j(C) \neq 0$ by construction. 

\subsection{Operator Interpretation of Kernel PCA}
\label{sec:pca_interpretation}

Kernel PCA approximates a finite-rank truncation of the \emph{Karhunen--Loève transformation}
of the embedded random variable $\varphi(X_0)$ by approximating the spectral decomposition
of the kernel covariance operator
$\cov = \Ex[ \varphi(X_0) \otimes \varphi(X_0) ]$~\citep[see for example][]{Blanchard2007}.

Consider the spectral decomposition~\eqref{eq:cov_spectral_decomposition}
of $C$.
Let $\{\tilde{v}_i\}_{i\geq 1}$ be an extension of the eigenfunctions 
$\{v_i\}_{i \in I}$ of $C$
to a complete orthonormal system in $\inrkhs$
(that is, the addition of an appropriate ONS spanning the null space of $C$)
By expanding the random variable $\varphi(X_0)$ in terms of 
$\{\tilde{v}_i\}_{i\geq 1}$, we get
the Karhunen--Loève transformation
\begin{equation} \label{eq:staticPCA}
\varphi(X_0) = 
\sum_{i \in I} \innerprod{\varphi(X_0)}{\tilde{v}_i}_\inrkhs \tilde{v}_i
 = \sum_{i \in I} Z_i \tilde{v}_i,
\end{equation}
where $Z_i := \innerprod{\varphi(X_0)}{\tilde{v}_i}_\inrkhs = \tilde{v}_i(X_0)$ are
real-valued random variables and convergence in~\eqref{eq:staticPCA}
is with respect to the norm of $\inrkhs$.
Note that we have
\begin{equation*}
    \Cov[Z_i, Z_j] = \Cov[(\tilde{v}_i(X_0), \tilde{v}_j(X_0)] 
= \innerprod{\tilde{v}_i}{\cov \tilde{v}_j}_{\inrkhs}  = \lambda_i(C) \delta_{ij},
\end{equation*}
where we extend the set of eigenvalues to
the null space, i.e., we set $\lambda_i(C): = 0$ for $i \neq I$.
In practice, the data is usually projected onto the 
first $r$ dominant eigenfunctions in order 
to obtain an optimal low-dimensional approximation of $\varphi(X_0)$.
In particular, for all $r \in I$, the projector
$P_{\leq r} := \sum_{i=1}^r v_i \otimes v_i$
minimizes the reconstruction error
\begin{equation}
    \label{eq:pca_error}
    R(T) := \Ex \left[ \norm{ \varphi(X_0) - T \varphi(X_0) }_\inrkhs^2 \right]
\end{equation}
over all operators $T$ in the set of $r$-dimensional orthogonal
projectors on $\inrkhs$.
By performing a spectral decomposition of the empirical kernel covariance 
operator 
\begin{equation*}
\cov[n] = \frac{1}{n} \sum_{t = 1}^n \varphi(X_t) \otimes \varphi(X_t),
\end{equation*}
kernel PCA aims to approximate~\eqref{eq:staticPCA}
(or $P_{\leq r}$ respectively).
We are therefore interested in how well the spectral decomposition
of the empirical operator $\cov[n]$ approximates the
spectral decomposition of $C$.

\subsection{Consistency Results}
We can now combine typical results
from spectral perturbation theory
with our previous error analysis for $\cov[n]$ 
to obtain consistency statements. Note that we do not aim to provide a full analysis
but rather illustrate how our results can be used to assess the error of kernel PCA with weakly
dependent data. In the independent case, stronger results have been obtained
for example by~\citet{Koltchinskii2016,Koltchinskii2017}, ~\cite{Milbradt2020}
and \citet{Reiss2020}
by directly considering~\eqref{eq:pca_error}.

\begin{remark}[Measurability of spectral properties]
As for example shown by~\cite{Dauxois1982}, the eigenvalues and corresponding 
eigenprojection operators of $\cov$ and $\cov[n]$ are measurable and therefore random variables
on $(\Omega, \sigalg, \Pr)$.
\end{remark}

\begin{theorem}[Spectral perturbation bounds]~\label{thm:spectral_perturbation}
    With the notation of Section~\ref{sec:pca_notation}, it holds that
   \begin{equation*}
       \sup_{i \geq 1}\abs{\lambda_i(\cov) - \lambda_i(\cov[n])} 
       \leq \norm{\cov - \cov[n]} \quad \Pr\textrm{-a.e.}
   \end{equation*}
   as well as
   \begin{equation*}
       \norm{P_j(\cov) - P_j(\cov[n])} 
       \leq \frac{4\norm{\cov - \cov[n]}}{g_j(\cov)}
       \quad \Pr\textrm{-a.e.}
   \end{equation*}
   for all $j \in J$.
\end{theorem}
See \citet[Corollary 2.3]{GohbergKrein} and~\citet[Lemma 1]{Koltchinskii2016} for
proofs of these statements.

The above bounds combined with the strong law of large numbers
from Corollary~\ref{cor:SLLN} for $\norm{\cov - \cov[n]}$ yield consistency results
of kernel PCA with weakly dependent data.

\begin{corollary}[Spectral consistency \& convergence rate]
    \label{cor:spectral_consistency}
    Let $(X_t)_{t \in \Z}$ be stationary and ergodic. Then 
    kernel PCA is strongly consistent in the sense that
    we have spectral convergence
    $\sup_{i \geq 1} \abs{\lambda_i(\cov) - \lambda_i(\cov[n])} \rightarrow 0$
    $\Pr$-a.e.\ as well as
    $\norm{P_j(\cov[n]) - P_j(\cov)} \to 0$
    $\Pr$-a.e.\ for all $j \geq 1$.
    In both cases, convergence takes place with \new{at least the same rate}
    as the convergence $\cov[n] \to \cov$ in operator norm.
\end{corollary}

\begin{remark}
    The preservation of convergence rates in 
    Corollary~\ref{cor:spectral_consistency} is particularly relevant
    whenever the assumptions of Corollary~\ref{cor:convergence_rate} hold.
    In this situation, the spectral convergence rate is given 
    by Corollary~\ref{cor:spectral_consistency}.
    We note that these results are by no means optimal,
    as they do not consider the full reconstruction error
    of a finite-rank truncation of~\eqref{eq:staticPCA}, like
    for example~\citet{Reiss2020} in the independent case.
    Stronger results can be obtained by accessing
    deeper perturbation results~\citep[see for example][]{Yu2014DavisKahan,Wahl2020}
    and are not in the scope of this work.
\end{remark}

Whenever the estimation error $\norm{\cov - \cov[n]}$ can be
bounded in probability (for example by applying 
Theorem~\ref{thm:finite_sample_bound})
corresponding statements hold for the eigenvalues and spectral projectors
as a result of Theorem~\ref{thm:spectral_perturbation}.

\begin{corollary}[Spectral concentration]
    Let $\Pr\left[\norm{C_n(0) - C(0)} \geq \epsilon \right] \leq f(\epsilon, n) $ for
    some function $f: \R_{>0} \times \N \rightarrow \R_{\geq 0}$.
    Then we have
    \begin{enumerate}
        \item $\Pr \left[ \sup_{i \geq 1} \abs{ \lambda_i(\cov) - \lambda_i(\cov[n] )} 
         \geq \epsilon \right] \leq f(\epsilon, n)$ and
     \item $\Pr \left[ \norm{ P_j(\cov) - P_j(\cov[n])}
         \geq \epsilon \right] \leq f(\frac{g_j(\cov)}{4} \epsilon , n)$ for all $j$.
    \end{enumerate}
\end{corollary}

\begin{remark}
    We note that in the case of weakly dependent data, the representation~\eqref{eq:staticPCA}
    might not always be a desirable model
    since time-related information in the realization of the process
    $(\varphi(X_t)_{t \in \Z})$ is discarded. As such,
    kernel PCA decomposes the RKHS only with respect to the covariance
    of $\law(\varphi(X_0)) $ instead
    of using autocovariance information from 
    $\law(\varphi(X_{t_1}),\varphi(X_{t_2}),\varphi(X_{t_3}),\dots)$.
    If one is interested in performing a decomposition
    that captures the \emph{dynamic behavior}
    instead of only the \emph{asymptotic spatial behavior},
    different approaches are needed.
    In the context of \emph{functional data analysis},
    the concept of \emph{harmonic PCA} or \emph{dynamic PCA}
    \citep{Panaretos2013a,HoermannDynamicFPCA} yields optimal 
    filter functions to reduce the dimensionality of a
    stationary stochastic process.
    We will address alternative time-based 
    decomposition approaches in Section~\ref{sec:markov_operators}.
\end{remark}

\section{Conditional Mean Embedding of Stationary Time Series}
\label{sec:CME}

We will now show how the previous theoretical results can be used to
obtain consistency results for a large family of
nonparametric time series models.
A wide variety of kernel techniques for sequential data rely on 
the RKHS embedding of the conditional $\lag$-time step transition probability
\begin{equation} \label{eq:transition_probability}
    \Pr[ X_{t + \lag} \in \mathcal{A} \mid X_t], \quad \mathcal{A} \in \sigalg_E,
    \quad \lag \in \N_{>0},
\end{equation}
which is modeled in terms of the \emph{conditional mean embedding}~\citep{SHSF09}.
In what follows, we will briefly outline the different derivations 
of the conditional mean embedding.

Applications of the conditional mean embedding in the context of 
sequential data include, among others, state-space models and 
filtering~\citep{SHSF09,Fukumizu13:KBR,Gebhardt2019Kalman},
the embedding of transition probability 
models~\citep{Songetal10,GrunewalderLBPGJ2012,Nishiyama2012,SunEtAl2019},
predictive state representations~\citep{BootsEtAl2013},
and reinforcement learning models~\citep{VanHoof15,VanHoof17,Stafford:ACCME18,Gebhardt2019Robot}.

\subsection{Operator-theoretic Conditional Mean Embedding}
\label{ssec:CME}
In order to express the transition probability~\eqref{eq:transition_probability} 
in terms of the RKHS~$\inrkhs$,
one is interested in a conditional mean operator 
$U_\lag \colon \inrkhs \supseteq\domain(U_\lag) \rightarrow \inrkhs$
which satisfies 
\begin{equation} \label{eq:CME_operator}
    \innerprod{f}
    { U_\lag \varphi(x)}_\inrkhs 
    = \Ex[f(X_{t + \lag}) \mid X_t = x ], \quad f \in \inrkhs.
\end{equation}
Note that the action of $U_\lag$ 
on $\varphi(x) \in \inrkhs$ 
is interpreted as conditioning on the event $\{ X_{t} = x \}$,
while evaluations of functions $f \in \inrkhs$ with 
$U_\lag \varphi(x)$ under the inner 
product can be interpreted as a 
conditional expectation operator in a weak sense.
It is important to note that such an operator $U_\lag$ does not
exist in general.
By using properties of the kernel covariance operators, it 
can be shown that 
$U_\lag := C(\lag)C(0)^\dagger$ satisfies \eqref{eq:CME_operator} 
under strong technical assumptions
(see~\citealp{Klebanov2019rigorous} for 
details).
We call $U_\lag$ 
the \emph{conditional mean operator} 
and $U_\lag\varphi(x)$ the 
\emph{conditional mean embedding} of the transition
probability 
$\Pr[X_{t + \lag} \in \mathcal{A} \mid X_t = x ]$.
Here, 
$C(0)^\dagger \colon \range(C(0)) \oplus \range(C(0))^\perp \rightarrow \inrkhs$ 
is the \emph{Moore--Penrose pseudoinverse} 
of the operator $C(0)$~\citep[see for example][]{EHN96}.
Note that $U_\lag$ is in general not globally defined and bounded,
i.e., $\range(C(0)) \oplus \range(C(0))^\perp \neq \inrkhs$,
since $\range(C(0))$ is generally not closed. 
\citet{SHSF09} propose the regularized
conditional mean operator
\begin{equation} \label{eq:CME_regularized}
    U_\lag^{(\gamma)} := C(\lag)\left( C(0) + \gamma \id_\inrkhs \right)^{-1},
\end{equation}
with the empirical estimate
$U_\lag^{(\gamma,n)} := C_n(\lag)\left( C_n(0) + \gamma \id_\inrkhs \right)^{-1}$.
Here, $\id_\inrkhs$ denotes the identity operator on $\inrkhs$ and
$\gamma > 0$ is a regularization parameter.
Note that $U_\lag^{(\gamma)}$ as well as $U_\lag^{(\gamma,n)}$ are 
always well-defined Hilbert--Schmidt operators on $\inrkhs$.
\citet{SHSF09}, \citet{Fukumizu13:KBR}, and \citet{Fukumizu15:NBI}
examine convergence of this estimate for the case of independent 
data pairs from the joint distribution of $X_t$ and $X_{t + \lag}$
and show weak consistency with different rates under various technical assumptions.
We extend these results to the case of dependent data.

Since the assumptions for this operator-theoretic framework
and especially the analytical existence of $U_\lag$
are hard to verify, different interpretations have emerged. 
In settings where~\eqref{eq:CME_operator} does not have an analytical solution,
the regularized estimate \smash{$U_\lag^{(\gamma,n)}$}
minimizes an empirical risk functional,
which we will briefly outline below.

\subsection{Least-squares Conditional Mean Embedding.}
In cases when $U_\lag$ is not globally defined and bounded, it is natural to
approximate a smooth solution to~\eqref{eq:CME_operator} by minimizing the
regularized
risk functional 
$\mathcal{E}_\lag^{(\gamma)} \colon \HS{\inrkhs} \rightarrow \R$ 
given by 
\begin{equation} \label{eq:CME_ansatz}
    \mathcal{E}_\lag^{(\gamma)}(A) := 
\sup_{\substack{f \in \inrkhs \\ \norm{f}_\inrkhs = 1}}
\Ex\left[ ( \Ex[f(X_{t + \lag}) \mid X_t] - 
\innerprod{f}{A\varphi(X_t)}_\inrkhs )^2 \right] + \gamma \norm{A}^2_{S_2(\inrkhs)},
\end{equation}
where $\gamma > 0$ is a regularization parameter.
As first shown by \citet{Gruen12}
and recently investigated by \citet{Park2020MeasureTheoretic}
and \citet{Mollenhauer2020Nonparametric},
the risk $\mathcal{E}_\lag^{(\gamma)}(A)$
can be bounded from above by the \emph{surrogate risk}
\begin{equation}\label{eq:CME_objective1}
    R_\lag^{(\gamma)}(A) := 
\Ex[ \norm{ \varphi( X_{t + \lag} ) - A\varphi(X_t) }_\inrkhs^2 ] + \gamma \norm{A}^2_{S_2(\inrkhs)},
\end{equation} \new{
which admits the minimizer 
$U_\lag^{(\gamma)}$ \citep{Mollenhauer2020Nonparametric}.}
The corresponding \emph{empirical surrogate risk}
\begin{equation}
    \label{eq:empirical_CME_risk}
    R_\lag^{(\gamma,n)}(A) := \frac{1}{n} \sum_{t = 1}^n 
    \norm{ \varphi(X_{t + \lag}) - A \varphi(X_t) }^2_\inrkhs
     + \gamma \norm{A}^2_{S_2(\inrkhs)}
 \end{equation} attains its minimum at 
 $U_\lag^{(\gamma,n)}$. 
\new{
We do not aim to cover all the
mathematical intricacies of the CME and its connection to least
squares regression here and
refer the reader to \citet{Park2020MeasureTheoretic} and
\citet{LiEtAl2022}
for more details.
}

\subsection{Kernel Sum Rule}
In the well-specified operator-theoretic setting of \eqref{eq:CME_operator},
\citet[][Theorem 2]{Fukumizu13:KBR} show that the conditional mean operator
$U_\lag$ satisfies the more general 
so-called \emph{kernel sum rule}, which is widely used
in nonparametric Bayesian models, especially time series filtering.
That is, for a \emph{prior} measure $z$ on $(E, \sigalg_E)$ 
satisfying the integrability $\int_E \norm{\varphi(Z)}_\inrkhs \dd z(Z) < \infty$ 
with a kernel mean embedding
$\mu_z = \int \varphi(Z) \dd z(Z)$ such that
$\mu_z \in \domain(C(0)^{\dagger})$,
we have
\begin{equation} \label{eq:kernel_sum_rule}
    \innerprod{f}{U_\lag\mu_z}_{\inrkhs} 
    = \int_{E}  \Ex[f(X_{t + \lag}) \mid X_t = x ] \, \dd z(x), \quad f \in \inrkhs.
\end{equation}
Note that the conditional mean property~\eqref{eq:CME_operator} is in fact a 
special case of the kernel sum rule
when $z$ is the Dirac measure at $x$, i.e., $\mu_z = \varphi(x)$.
In applications, the embedded prior 
$\mu_z$ is usually estimated empirically by sampling from $z$. 
When $\emu_z$ is any kind of consistent estimate of $\mu_z$, we
obtain the plug-in estimator $U_\lag^{(\gamma,n)} \emu_z$ for $U_\lag\mu_z$.

\subsection{Consistency Results}
We now outline how our previous results allow to formulate consistency results for
the kernel sum rule for dependent data.
Prior consistency results for the 
operator-based setting
of the conditional mean embedding and the kernel sum rule are limited to
independent data pairs. 
Note again that a drawback of our approach is the typical
assumption that the analytic expression 
$U_\lag\mu_z$ (and in particular $U_\lag\varphi(x)$) exists in
$\inrkhs$~\citep[see][]{Klebanov2019rigorous}, while a focus on the minimization
properties allows to relax this assumption and consider convergence to a 
best approximation under the corresponding risk.
We start by giving a generic error 
decomposition for the kernel sum rule in a form that 
admits the immediate application of our previous results.

\begin{theorem}[Kernel sum rule error] \label{thm:kernel_sum_rule_error}
    Let $z$ be a prior finite measure on $(E, \sigalg_E)$ and a 
    kernel $k \colon E \times E \rightarrow \R$ with 
    $\sup_{x \in E} k(x,x) = c < \infty$
    such that the kernel sum rule~\eqref{eq:kernel_sum_rule} 
    applies 
    \new{%
       for a fixed time lag $\eta \in \N_{>0}$
    }
    (in particular, $\mu_z \in \domain(C(0)^\dagger)$).
    Then the empirical estimate $U_\lag^{(\gamma,n)} \emu_z$ 
    admits the total error bound
    \begin{equation}\label{eq:kernel_sum_rule_error}
        \norm{U_\lag^{(\gamma,n)} \emu_z - U_\lag \mu_z}_\inrkhs 
        \leq e_s(\mu_z,\emu_z,n, \gamma) + 
        e_r(\mu_z,\gamma) \quad \Pr\textrm{-a.e.}
    \end{equation}
    with the stochastic \emph{estimation error}
    \begin{equation*}
    e_s(\mu_z,\emu_z,n, \gamma) := \frac{c}{\gamma} \norm{\emu_z - \mu_z}_\inrkhs +
    \frac{c^{3/2}}{\gamma^2} \norm{C_n(0) - C(0)} +
    \frac{c^{1/2}}{\gamma} \norm{C_n(\lag) - C(\lag) }
    \end{equation*}
    and the deterministic \emph{regularization error}
    \begin{equation*}
    e_r(\mu_z,\gamma) := c \norm{(C(0) + \gamma \idop[\inrkhs])^{-1} \mu_z
        - C(0)^\dagger \mu_z}_\inrkhs.
    \end{equation*} 
\end{theorem}
The proof for Theorem~\ref{thm:kernel_sum_rule_error} can be
found in Appendix~\ref{app:proofs}.
\new{%
In the context of inverse problems,
the estimation error $e_s$ 
is sometimes also called the \emph{sample error},
while the regularization error is
sometimes also called the \emph{approximation error}.
}

\begin{remark}[Kernel sum rule error]
    \label{rem:sum_rule_error}
    The error decomposition~\eqref{eq:kernel_sum_rule_error} leads to the following insights:
    \begin{enumerate}[label=(\roman*)]
        \item \label{rem:sum_rule_error_regularization}
        The deterministic regularization error $e_r(\mu_z,\gamma)$ 
        captures the analytic nature of the
        inverse problem described by $C(0)u = \mu_z$ for $u \in \inrkhs$.
        As such, it is not affected by any estimation.
        Note that as $\gamma \to 0$, it holds that 
        $e_r(\mu_z,\gamma) \rightarrow 0$ for every
        $\mu_z \in \domain(C(0)^\dagger)$.  
        For details, we refer the reader to~\citet{EHN96}.
        In practice, regularization is needed since for estimated
        right-hand sides $\emu_z$, the condition 
        $\emu_z \in \domain(C(0)^\dagger)$ 
        is in general not true---even if 
        $\mu_z \in \domain(C(0)^\dagger)$.
        The convergence rate of $e_r(\mu_z,\gamma)$
        depends on the eigendecomposition of $C(0)$ and can be 
        assessed under additional assumptions about the decay rate of the eigenvalues.
        However, in general the convergence of the regularization error 
        can be arbitrarily slow without 
        additional assumptions, see~\citet{Schock1984}.
        \item For convergence of the total error~\eqref{eq:kernel_sum_rule_error},
        we need the two simultaneous conditions $e_s(\mu_z,\emu_z,n, \gamma) \to 0$
        and $e_r(\mu_z,\gamma) \to 0$ as $n \to \infty$, $\gamma \to 0$ and
        $\emu_z \to \mu_z$.
        The typical trade-off between regularization error
        and estimation error is reflected in this fact.
        Our previous convergence results for the individual estimation errors
        of $\cov[n](0)$ and $\cov[n](\lag)$ allow to bound $e_s(\mu_z,\emu_z,n, \gamma)$
        and derive regularization schemes
        $\gamma:=\gamma(n,\mu_z,\emu_z)$ depending on the trajectory length $n$
        and the quality of the prior estimate $\emu_z$.
        Informally speaking, the individual estimation errors
        must tend to $0$ faster than the regularization term, so $\gamma$ should
        not be allowed to converge ``too fast'' with respect to 
        the rate of increasing sample size $n$ -- this is the typical
        setting in the theory of inverse problems and regularization.
          \end{enumerate}
\end{remark}

By incorporating additional knowledge about the convergence behavior of 
$\cov[n](0)$ and $\cov[n](\lag)$ from our previous results, Theorem~\ref{thm:kernel_sum_rule_error}
yields convergence rates of the estimation error $e_s$ as well
as admissible regularization schemes. We give an example below.
For simplicity, we assume that $\cov[n](0)$ and $\cov[n](\lag)$ are estimated independently,
which would of course require two realizations of length $n$ of $(X_t)_{t \in \Z}$.
To illustrate the idea, we require $\emu_z$ to converge with a $\Pr$-a.e.\ rate of
$1/\sqrt{n}$, which we tie to the number of samples available for the estimation of
$\cov[n](0)$ and $\cov[n](\lag)$ in order to avoid additional symbols for different samples.
Our general approach presented here 
still applies when the convergence rate of $\emu_z$ is slower.

\begin{example}[Kernel sum rule consistency]
    \label{ex:sum:rule}
    Let $C_n(0)$ and $C_n(\lag)$ be estimated independently
    from samples $X_1, \dots, X_n$.
    Assume that we have the prior convergence rate
    \begin{equation*}
        \norm{\mu_z - \emu_z}_\inrkhs = O(1 / \sqrt{n})
        \quad \Pr\textnormal{-a.e.}
    \end{equation*}
    \rev{%
    Then under the conditions of Theorem~\ref{thm:clil},}
    we have
    \begin{equation*}
        e_s(\mu_z,\emu_z,n, \gamma) 
        \rev{%
        = O
        \left( \frac{ \sqrt{ 2 L(L(n) } )}{ \sqrt{n} \gamma^2} 
        \right)
        }
        \quad \Pr\textnormal{-a.e.}
   \end{equation*}
   In particular, for every regularization scheme $\gamma = \gamma(n)$
   such that 
   \begin{equation*}
       \gamma(n) \to 0 \textrm{ as well as }
       \rev{%
       \frac{ \sqrt{2 L(L(n))}}{ \sqrt{n} \gamma(n)^2} \to 0
       }
        \quad \Pr\textnormal{-a.e.}
   \end{equation*}
   for $n \to \infty$,
   we have the overall convergence
   \begin{equation*}
       U_\lag^{(\gamma(n),n)} \emu_z \to U_\lag\mu_z 
        \quad \Pr\textnormal{-a.e.}
   \end{equation*}
    in the norm of $\inrkhs$.
   \new{%
    Note that this follows 
    since we have already covered the convergence
    of the regularization error 
    $e_r(\mu_z,\gamma(n)) \to 0$ in 
   Remark~\ref{rem:sum_rule_error}\ref{rem:sum_rule_error_regularization}.
   }
\end{example}

\begin{remark}[Counterfactual mean embedding]
\new{%
    The results highlighted in this section
    require the fairly restrictive assumptions of the kernel sum rule
    as originally formulated by \cite{Fukumizu13:KBR}.
    As \citet{Muandet2021} have recently shown,
    the kernel sum rule estimator can
    be interpreted in the context of the so-called 
    \emph{counterfactual mean embedding},
    where a convergence analysis can be based
    on the convergence of empirical kernel covariance operators
    under significantly weaker assumptions.
    Our results from the previous sections
    can be readily applied to the
    convergence analysis presented by \citet[Section 4]{Muandet2021},
    leading to non-i.i.d.\ results in the framework
    of causal inference.
    }
\end{remark}

\section{Nonparametric Estimation of Markov Transition Operators}
\label{sec:markov_operators}

As the last application of our theory, we will briefly show how
the risk functional~\eqref{eq:CME_ansatz} yields a nonparametric 
model for the estimation of 
\emph{Markov transition operators}.
Moreover, we elaborate on the recent discovery that this model is
actually the theoretical foundation of a well-known
family of several data-driven methods
for the analysis of dynamical systems~\citep{Klus2019, Mollenhauer2020Nonparametric}.
We only highlight immediate consequences of this approach and 
emphasize that several theoretical questions need to be 
answered separately in a vector-valued 
statistical learning context \citep{Park2020MeasureTheoretic, Mollenhauer2020Nonparametric}.
The aim of this section is to draw attention to the fact that
statistical tools like strong mixing coefficients can
be used to show consistency for a range of numerical methods used
in other scientific disciplines.
\new{%
For the mathematical background of discrete-time Markov processes
and Markov transition operators, we refer the reader to \citet{Douc2018Markov}.
}

In what follows, we assume 
that $(X_t)_{t \in \Z}$ is a stationary Markov process, i.e., it holds
\begin{equation*}
\Ex\left[ f(X_{s}) \mid \sigalg_{- \infty}^t \right] = 
\Ex\left[ f(X_{s}) \mid \sigma(X_t) \right]
\end{equation*}
for all bounded measurable functions $f \colon E \rightarrow \R$ and
times $s \geq t$.
For a fixed time lag $\lag \in \N_{>0}$, 
the \emph{transition operator}, \emph{(backward) transfer operator}\footnote{The name
    \emph{backward transfer operator} 
    is classically used in the context
    of continuous-time processes, where it is used to describe
    the solution to the 
    \emph{backwards Kolmogorov equation}.
    In the theory of dynamical systems, the
    term \emph{Koopman operator} is commonly used.} or 
\emph{(stochastic) Koopman operator} $\ko_\lag$ is defined by the relation
\begin{equation} \label{eq:koopman}
(\ko_\lag f)(x) = \Ex[f(X_{t + \lag}) \mid X_t = x] 
\end{equation}
for all functions $f: E \rightarrow \R$ in some 
appropriately chosen subspace $F$ of
measurable functions.
It 
describes the propagation of observable functions in $F$ by the time step
$\lag$ under the dynamics given by $(X_t)_{t \in \Z}$.

\new{
The Markov transition operator is a fundamentally important tool 
for the analysis of various properties of Markov processes,
Markov chain Monte Carlo methods, and dynamical systems.
For instance, it is known that the
spectrum of $\ko_\lag$ and the associated eigenfunctions determine
a crucial set of related properties of the underlying dynamics such as ergodicity, speed of
convergence, the decomposition of the state space into almost
invariant components 
(so-called \emph{metastable states}), 
several contraction and concentration results and many more
(see \citealt{MeTw12, Rudolf2012Diss,Bovier2016Metastability, Douc2018Markov} 
and the literature reviews therein).
}

By simply switching to
the adjoint of $A$ in the expression for $\mathcal{E}_\lag^{(\gamma)}(A)$ 
defined in~\eqref{eq:CME_ansatz} and using
the reproducing property of $\inrkhs$, we have
\begin{equation*} 
    \mathcal{E}_\lag^{(\gamma)}(A) 
    := \sup_{\substack{f \in \inrkhs \\ \norm{f}_\inrkhs = 1}}
    \Ex\left[ ( \Ex[f(X_{t + \lag}) \mid X_t] - 
(A^*f)(X_t) )^2 \right] + \gamma \norm{A^*}^2_{S_2(\inrkhs)}.
\end{equation*} 
As a result, we can immediately interpret the
adjoint of the conditional mean operator
\begin{equation*} 
    U_\lag^{(\gamma)*} = 
 \left( C(0) + \gamma \id_\inrkhs \right)^{-1} C(\lag)^*
\end{equation*}
as a smooth approximation of the transition operator $\ko_\lag$ on
the class of RKHS functions $F = \inrkhs$
with empirical estimate
\begin{equation*}
U_\lag^{(\gamma,n)*} = 
\left( C_n(0) + \gamma \id_\inrkhs \right)^{-1} C_n(\lag)^*.
\end{equation*}
Note that all of our consistency results for kernel autocovariance
operators and the conditional mean embedding transfer directly to 
this setting, as operator norm error bounds
for the estimate of $U_\lag^{(\gamma)}$ are also valid for its adjoint.
The approximation-theoretic details
of this approach are investigated by \citet{Mollenhauer2020Nonparametric}.

The idea of approximating the operator $\ko_\lag$
via $U_\lag^{(\gamma)*}$ can be connected to 
data-driven spectral analysis techniques 
which are commonly used in  fluid dynamics, 
molecular dynamics, and atmospheric sciences. 
One of the most widely used spectral analysis and 
methods is the so-called
\emph{extended dynamic mode decomposition}~(EDMD,\citealp{WKR15}).
EDMD computes Galerkin approximation of the eigendecomposition of $\ko_\lag$
based on a finite set of basis functions in $F$.
When $F$ is chosen to be the RKHS $\inrkhs$,
one obtains \emph{kernel EDMD}~\citep{WRK15:Kernel} as a special 
nonparametric version of EDMD.
It was shown by~\citet{Klus2019} that regularized kernel EDMD 
actually computes the eigendecomposition 
of~$\smash{U_\lag^{(\gamma,n)*}}$.
In contrast to previous results that rely on 
ergodicity of the underlying 
system~\citep{KKS16,KoMe18}, 
our results may lead to a refined convergence analysis
by using mixing properties of the underlying system.
To the best of our knowledge, prior consistency results for EDMD
only cover convergence in
strong operator topology 
(i.e., pointwise convergence) for parametric models, i.e. 
on fixed finite-dimensional subspaces
spanned by a dictionary of basis functions.
Furthermore, they mostly aim towards deterministic dynamical
systems~\citep{KoMe18}. 
However, we note that the operator 
\smash{$U_\lag^{(\gamma)*}$}
is in general not self-adjoint and a dedicated analysis
of spectral properties and convergence is subject to future work.

Additionally, it is known that the operator $\ko_\lag$ and its adjoint,
the so-called \emph{Perron--Frobenius} operator, can be connected
to the solution of the so-called 
\emph{blind source separation problem}~\citep{KNKWKSN18,Klus2019}.
In fact, eigenfunctions of compositions of empirical
autocovariance operators (and their pseudoinverses)
are used as projection coordinates in a kernel-based variant of 
\emph{independent component analysis}~\citep{HZKM03,SP15}.
As such, consistency results formulated by us 
may be used to prove convergence for these approaches.

\section{Conclusion}
\label{sec:conclusion}
In this paper, we provided a mathematically rigorous analysis of
kernel autocovariance operators and established classical limit theorems
as well as nonasymptotic error bounds under classical ergodic and
mixing assumptions. The results were mostly derived from theoretical 
work on discrete-time processes in
Hilbert spaces and are presented in a form such that they
can be easily applied in the context of RKHS-based time series models
and frequency domain analysis.
We highlighted high-level applications for kernel PCA,
the conditional mean embedding, and
the nonparametric estimation of Markov transition operators.
The theory of vector-valued statistical learning
from dependent data may be connected to our considerations in future work.
In the context of learning Markov transition operators, the
kernel autocovariance operator may lead to an inverse problem 
that describes the analytical
minimizer of an autoregression risk in an operator space.

\section*{Acknowledgements}
	We greatfully acknowledge funding from Germany’s  Excellence  Strategy  (MATH+: The Berlin Mathematics Research Center, EXC-2046/1, project ID:390685689, project AA1-2).
    SK and PK are partially supported by the Deutsche Forschungsgemeinschaft (DFG), CRC1114, projects B06 and A01.
    The authors wish to thank Ilja Klebanov for his helpful comments,
    \new{and the anonymous reviewers who helped to improve the manuscript}.

\appendix

\section{Proofs}
\label{app:proofs}

We report the proofs which were omitted in the main text.

\subsection{Proof of Theorem~\ref{thm:clil}}
    \label{proof:clil}
    We apply Theorem 2 of~\citet{Merlevede2008:LIL} 
    to the process $(\xi)_{t \in \Z}$
    \new{%
        defined by \eqref{eq:process_shorthand}
        for the fixed time lag $\lag \in \N$.
    }
    This results verifies the existence of
    a compact set $K_\lag$ with the desired properties such that
    both~\eqref{eq:LIL-lim} and~\eqref{eq:LIL-acc} hold.
    \new{%
    We note that \citet[Remark 3]{Merlevede2008:LIL}
	ensures that our assumptions allow the application of this result.
    }
    It now remains to show the norm bound~\eqref{eq:LIL-bound} for $K_\lag$.
    The set $K_\lag$ is the unit ball of the Hilbert space 
    $\mathbb{H}_\lag \subseteq S_2(\inrkhs)$, 
    which is given by the completion of the range of 
    $T_\lag^{1/2}$ 
    (where $T_\lag$ is given by~\eqref{eq:asymptotic_cov_operator} and
    $T_\lag^{1/2}$ denotes its operator square root)
    with respect to the inner product defined by
    \begin{equation} \label{eq:induced-rkhs-innerprod}
        \innerprod{T_\lag^{1/2}A}{T_\lag^{1/2}B}_{\mathbb{H}_\lag} 
        := \innerprod{A}{B}_{\HS{\inrkhs}},\quad A,B \in \HS{\inrkhs}.
    \end{equation}
    The space $\mathbb{H}_\lag$ is also called \emph{Cameron--Martin space} or 
    \emph{abstract Wiener space}
    \new{in the context of Gaussian measures}~\citep%
    [for details, we refer the reader to][Chapter 2]{BogachevGaussianMeasures}.
    \new{%
        It is noteworthy that
        $\mathbb{H}_\lag$ itself is sometimes also called
        \emph{reproducing kernel Hilbert space}
        (associated with the so-called \emph{covariance kernel}
        defined by $T_\lag$, see \citealp{Merlevede2008:LIL}, Theorem 2),
        which may seem misleading in our context
        at first glance.
        In particular, note that the space
        $S_2(\inrkhs)$ consists of
        Hilbert--Schmidt operators and not of real-valued functions
        as in our definition of an RKHS in this paper.
        The interpretation of $\mathbb{H}_\lag$ as a reproducing kernel
        Hilbert space requires an identification of
        $\mathbb{H}_\lag$ with its dual space
        as real-valued functions contained
        in the space $L^2(\mathcal{N}(0,T))$
        as described by
        \citet[Remark 2.2.3]{BogachevGaussianMeasures}.
        The final connection between this setting
        and our definition of an RKHS
        is made clear by Mercer's theorem and the related theory of
        integral operators \citep[Chapter 4.5]{StCh08}.
        We will not cover all the mathematical details here and
        refer the reader to more 
        comprehensive treatments below.
    }

    For a technical construction of $\mathbb{H}_\lag$ and the limit set $K_\lag$ 
    in the law of the iterated logarithm in Banach spaces, 
    we refer the reader to~\citet[][Section 2]{Kuelbs1976} 
    as well as~\citet[][Section 2]{Goodman1981}.
    Note that these references
    elaborate on the i.i.d.\ case. 
    However, for the construction of $\mathbb{H}_\lag$ and $K_\lag$ 
    only an abstract limiting probability measure is needed,
    which is given by the Gaussian measure obtained from Theorem~\ref{thm:CLT}
    and its covariance operator $T_\lag$ 
    defined by~\eqref{eq:asymptotic_cov_operator},
    just as shown in the proof 
    of~\citet[][Theorem~2]{Merlevede2008:LIL}.
    We can therefore analyze properties of
    $K_\lag$ by considering the Cameron--Martin space of the 
    centered Gaussian measure induced by $T_\lag$, which is examined in the
    previously mentioned literature.
    The identity~\eqref{eq:induced-rkhs-innerprod} 
    can be verified by translating the
    abstract Banach space definition of
    \cite[][Equation 2.3]{Kuelbs1976} to our scenario of the 
    separable Hilbert space $\HS{\inrkhs}$ 
    as, for example, 
    described by~\citet[][Remark 2.3.3]{BogachevGaussianMeasures}.
    
    From~\eqref{eq:induced-rkhs-innerprod}, we obtain 
    \begin{equation} 
        \label{eq:LIL-RKHS-norm-bound}
        \norm{A}_{\HS{\inrkhs}} 
        \leq \norm{T_\lag^{1/2}} \norm{A}_{\mathbb{H}_\lag}, \quad A \in \mathbb{H}_\lag.
    \end{equation} 
    Since $K_\lag 
    = \{ A \in \mathbb{H}_\lag \mid {\norm{A}_{\mathbb{H}_\lag}} \leq 1 \}$,
    a bound for $\norm{T_\lag^{1/2}} = \norm{T_\lag}^{1/2}$ 
    depending on the mixing rate of $(\xi_t)_{t \in \Z}$ is sufficient
    in order to provide a bound for elements of $K$ in the norm of $\HS{\inrkhs}$. 
    
    We now give a norm bound for
    $T_\lag = \Ex[ \xi_0 \otimes \xi_0 ] + \sum_{t = 1}^\infty \Ex[ \xi_0 \otimes \xi_t ] +
    \sum_{t = 1}^\infty \Ex[ \xi_t \otimes \xi_0 ] $.
    We clearly have 
    \begin{equation*}
        \norm{\Ex[ \xi_0 \otimes \xi_0 ]} \leq 4c^2,  
    \end{equation*} since $\xi_0$ is almost surely bounded by $2c$
    by~\eqref{eq:bounded-process}.
   
    Let $\alpha(n)$ be the mixing coefficients of $(X_t)_{t \in \Z}$.
    We now note that by~\eqref{eq:mixing_transformation2}, we have 
    $\alpha((\xi_t)_{t \in \Z}, n) \leq \alpha(n-\lag)$
    for all $n \in \N$.
    This allows to give a bound for the two remaining summands of $T$: 
    \begin{align*}
        \norm { \sum_{t = 1}^\infty \Ex[ \xi_t \otimes \xi_0 ] }
        &\leq \sum_{t = 1}^\infty  \norm { \Ex[ \xi_t \otimes \xi_0 ] } \\
        &=  \sum_{t = 1}^\infty \sup_{ \substack{\norm{A}_{\HS{\inrkhs}} 
        = 1 \\ \norm{B}_{\HS{\inrkhs}} = 1}}
        \abs{ \Ex[ \innerprod{\xi_t}{B} \innerprod{\xi_0}{A} ] } \\
        &\leq \sum_{t = 1}^\infty \sup_{ \substack {\norm{A}_{\HS{\inrkhs}} = 1 \\ 
        \norm{B}_{\HS{\inrkhs}} = 1}}
        4 \, \alpha(\sigma(\xi_t), 
            \sigma(\xi_0)) \norm{\innerprod{\xi_t}{B}}_{L^\infty(\Pr)}
        \norm{\innerprod{\xi_0}{A}}_{L^\infty(\Pr)} \\
        &\leq \sum_{t = 1}^\infty 16\, c^2 \alpha(t - \lag)  = 16\, c^2 M,
    \end{align*}
    where we use Ibragimov's covariance inequality for strongly mixing and
    bounded random variables \citep[Lemma 1.2]{Ibragimov1962} in the third step (note 
    that $\innerprod{\xi_t}{B}$ and $\innerprod{\xi_0}{A}$ are centered real-valued
    random variables which are $\Pr$-a.e.\ bounded by $2c$
    because of Equation~\ref{eq:bounded-process}).
    By symmetry, we obtain the same bound for 
    $\norm { \sum_{t = 1}^\infty \Ex[ \xi_0 \otimes \xi_t ] }$
    and we end up with the total norm bound
    \begin{equation*}
        \norm{T_\lag} \leq 4c^2 + 32\, c^2 M ,
    \end{equation*} which proves the claim 
    in combination with~\eqref{eq:LIL-RKHS-norm-bound}.
    \QED

\subsection{Proof of Theorem~\ref{thm:kernel_sum_rule_error}}

    Note that since $\sup_{x \in E} k(x,x) = c < \infty$,
    we have the $\Pr$-a.e.\ bounds $\norm{\mu_z}_\inrkhs \leq c^{1/2}$ as well as $\norm{C(\lag)} \leq c$
    for all $\lag \in \N$.
    Additionally, the regularized inverse
    can be bounded as $\norm{(C(0) + \gamma \id_\inrkhs)^{-1}} \leq \frac{1}{\gamma}$, 
    which is easy to see from
    the corresponding spectral decomposition.
    These bounds hold analogously for the empirical versions of all above objects.\!\footnote{For
    $\emu_z$, estimators of the form 
    $\emu_z := \sum_i \beta_i \varphi(x_i)$ with coefficients $\sum_i \abs{\beta_i} = 1$ 
    naturally satisfy the bound.}
    All following bounds below will be understood in the $\Pr$-a.e.\ sense for the remainder of
    this proof.
    
    We now successively insert appropriate zero-sum terms into the total error 
    and apply the triangle inequality multiple times to obtain the worst-case estimation error.
    We have the overall decomposition        
    \begin{align}
        \norm{U_\lag^{(\gamma,n)} \emu_z - U_\lag\mu_z}_\inrkhs \leq
        \underbrace{
            \norm{U_\lag^{(\gamma,n)} \emu_z -U_\lag^{(\gamma,n)} \mu_z}_\inrkhs
        }_{(I)}
        + 
        \underbrace{
            \norm{U_\lag^{(\gamma,n)} \mu_z - U_\lag\mu_z}_\inrkhs
        }_{(II)}.
    \end{align}
    For these two error components, we get the individual bounds
    \begin{align*}
    (I) \leq \norm{C_n(\lag)} 
    \norm{(C_n(0) + \gamma \id_\inrkhs)^{-1}(\emu_z - \mu_z) }_\inrkhs
    \leq \frac{c}{\gamma}  
    \norm{\emu_z - \mu_z}_\inrkhs
    \end{align*}
    as well as
    \begin{align*}
    (II) 
    &\leq 
    \underbrace{
        \norm{
            C_n(\lag) (C_n(0) + \gamma \id[\inrkhs])^{-1} \mu_z
            -  C_n(\lag)  (C(0) + \gamma \id[\inrkhs])^{-1} \mu_z
        }_\inrkhs
    }_{(\star)}
    \\ &+ 
    \underbrace{
        \norm{
            C_n(\lag) (C(0) + \gamma \id[\inrkhs])^{-1} \mu_z
            - C(\lag) C(0)^\dagger \mu_z
        }_\inrkhs
    }_{(\star \star)}.
    \end{align*}
     For $(\star)$, we give a bound by 
     \begin{align*}
         (\star) &\leq 
     \norm{C_n(\lag)}  \norm{(C_n(0) + \gamma \id[\inrkhs])^{-1}
         - ( C(0) + \gamma \id[\inrkhs])^{-1} } \norm{\mu_z}_\inrkhs \\
                 &\leq \frac{c^{3/2}}{\gamma^2} \norm{C_n(0) - C(0)},
     \end{align*}
     where we use the identity 
     $A^{-1} - B^{-1} = A^{-1} (B - A) B^{-1}$ for
     invertible operators $A$ and $B$.
     To obtain a bound for $(\star \star)$, we again insert a zero-sum term:
    \begin{align*}
    (\star \star) &\leq \norm{C_n(\lag) (C(0) + \gamma \idop[\inrkhs])^{-1}\mu_z
        - C(\lag) ( C(0) + \gamma \idop[\inrkhs])^{-1}\mu_z }_\inrkhs 
    \\&{\phantom{=}} +
    \norm{C(\lag)(C(0) + \gamma \idop[\inrkhs])^{-1} \mu_z
        - C(\lag)C(0)^\dagger \mu_z}_\inrkhs
    \\&\leq
    \norm{C_n(\lag) - C(\lag) }
    \norm{(C(0) + \gamma \idop[\inrkhs])^{-1}} \norm{\mu_z }_\inrkhs
    \\&{\phantom{=}} +
    \norm{C(\lag)} \norm{(C(0) + \gamma \idop[\inrkhs])^{-1} \mu_z
        - C(0)^\dagger \mu_z}_\inrkhs
    \\
    & \leq \frac{c^{1/2}}{\gamma} \norm{C_n(\lag) - C(\lag) } + c\norm{(C(0) + \gamma \idop[\inrkhs])^{-1} \mu_z
        - C(0)^\dagger \mu_z}_\inrkhs.
    \end{align*} 
    The sum of the bounds $(I)$, $(\star)$, and $(\star \star)$ yields the total bound as given
    in~\eqref{eq:kernel_sum_rule_error} after rearranging.
    \QED

\bibliography{library}

\end{document}